\documentclass[11pt,reqno,twoside]{amsart}
\usepackage{tikz}
\usepackage{amssymb}
\usepackage{amsmath}
\usepackage{amsthm}
\usepackage{amsfonts}
\usepackage{amsthm,amscd}
\usepackage{amsbsy}
\usepackage{bm}
\usepackage{url} 
\usepackage{setspace}
\usepackage{float}
\usepackage{psfrag}
\usepackage{tikz}
\usetikzlibrary{calc,intersections}
\usepackage{float}
\usepackage{epstopdf}
\usepackage{caption}
\usepackage{subcaption}
\usepackage{graphicx}
 \usepackage{caption} 
\allowdisplaybreaks

\usetikzlibrary{positioning}

\usepackage[colorlinks, linkcolor=blue]{hyperref}

\restylefloat{figure}

\newtheorem{Theorem}{Theorem}[section]
\newtheorem{Lemma}[Theorem]{Lemma}

\newtheorem{Proposition}[Theorem]{Proposition}
\newtheorem{Remark}[Theorem]{Remark}

\numberwithin{equation}{section}
\usepackage{setspace}
\usepackage{anysize}

\usepackage[margin=1in]{geometry}
\usepackage[noadjust]{cite}

\def\be{\begin{equation}}
	\def\ee{\end{equation}}
\def\ben{\begin{eqnarray}}
	\def\een{\end{eqnarray}}

\hypersetup{colorlinks=true,%
	citecolor=red,%
	filecolor=blue,%
	linkcolor=blue,%
}

\newcommand{\ncom}{\newcommand}
\usepackage{mathtools}
\mathtoolsset{showonlyrefs=true}

\ncom{\n}{\normalfont}

\ncom{\N}{\mathbb{N}}
\ncom{\Lc}{\mathcal}
\ncom{\wt}{\widetilde}
\ncom{\Wf}{\boldsymbol{w}}
\ncom{\Af}{\boldsymbol{A}}
\ncom{\Bf}{\boldsymbol{B}}
\ncom{\Hf}{\boldsymbol{H}}
\ncom{\Pf}{\boldsymbol{P}}
\ncom{\no}{\nonumber}
\ncom{\ub}{\boldsymbol{u}}
\ncom{\yb}{\boldsymbol{y}}
\ncom{\vb}{\boldsymbol{v}}
\ncom{\T}{\mathbb{T}}
\ncom{\C}{\mathbb{C}} 
\ncom{\Hb}{\mathbb{H}}
\ncom{\Lb}{\mathbb{L}}
\ncom{\V}{\mathbb{V}}
\ncom{\U}{\mathbb{U}}
\ncom{\Ac}{\mathcal{A}}
\ncom{\Bc}{\mathcal{B}}
\ncom{\af}{\boldsymbol{a}}
\ncom{\pf}{\boldsymbol{p}}
\ncom{\zb}{\boldsymbol{z}}
\newcommand{\vertiii}[1]{{\left\vert\kern-0.25ex\left\vert\kern-0.25ex\left\vert #1 
		\right\vert\kern-0.25ex\right\vert\kern-0.25ex\right\vert}}

\long\def\/*#1*/{}
\title[Stabilizability of a generalized Burgers-Huxley equation]{Feedback Stabilizability of a generalized Burgers-Huxley equation with kernel  around a non-constant steady state}
\date{\today}

\author{{WASIM AKRAM$^\dag$, Manika Bag$^\ddag$,}
	\and{Manil T. Mohan$^\dag$}}

\thanks{$\dag$ Department of Mathematics, Indian Institute of Technology Roorkee, Uttarakhand, 247667, India, Email- {\normalfont{ wakram2k11@gmail.com; maniltmohan@ma.iitr.ac.in}}\\
	$\ddag$ Department of Mathematics, Indian Institute of Science, Education and Research, Thiruvananthapuram,  Kerala, 695551, India, Email- {\normalfont {manikabag19@iisertvm.ac.in}}\\
	Dr. Wasim is supported by NBHM (National Board of Higher Mathematics, Department of Atomic Energy) postdoctoral fellowship, No. 0204/16(1)(2)/2024/R\&D-II/10823. Ms. Manika Bag would like to thank IIT Roorkee for providing stimulating scientific environment and resources.}

\makeatletter
\renewcommand{\tocsection}[3]{%
	\indentlabel{\@ifnotempty{#2}{\bfseries\ignorespaces#1 #2\quad}}\bfseries#3}
\renewcommand{\tocsubsection}[3]{%
	\indentlabel{\@ifnotempty{#2}{\ignorespaces#1 #2\quad}}#3}

\newcommand\@dotsep{4.5}
\def\@tocline#1#2#3#4#5#6#7{\relax
	\ifnum #1>\c@tocdepth 
	\else
	\par \addpenalty\@secpenalty\addvspace{#2}%
	\begingroup \hyphenpenalty\@M
	\@ifempty{#4}{%
		\@tempdima\csname r@tocindent\number#1\endcsname\relax
	}{%
		\@tempdima#4\relax
	}%
	\parindent\z@ \leftskip#3\relax \advance\leftskip\@tempdima\relax
	\rightskip\@pnumwidth plus1em \parfillskip-\@pnumwidth
	#5\leavevmode\hskip-\@tempdima{#6}\nobreak
	\leaders\hbox{$\m@th\mkern \@dotsep mu\hbox{.}\mkern \@dotsep mu$}\hfill
	\nobreak
	\hbox to\@pnumwidth{\@tocpagenum{\ifnum#1=1\bfseries\fi#7}}\par
	\nobreak
	\endgroup
	\fi}
\AtBeginDocument{%
	\expandafter\renewcommand\csname r@tocindent0\endcsname{0pt}
}
\def\l@subsection{\@tocline{2}{0pt}{2.5pc}{5pc}{}}
\makeatother

\begin{document}
	
	\pagenumbering{arabic}

\begin{abstract}
In this article, we investigate a generalized Burgers-Huxley equation with a smooth kernel defined in a bounded domain $\Omega\subset\mathbb{R}^d$, $d\in\{1,2,3\}$, focusing on feedback stabilizability around a non-constant steady state. Initially, employing the Banach fixed point theorem, we establish the local existence and uniqueness of a strong solution, which is subsequently extended globally using an energy estimate. To analyze stabilizability, we linearize the model around a non-constant steady state and examine the stabilizability of the principal system. For the principal system, we develop a feedback control operator by solving an appropriate algebraic Riccati equation. This allows for the construction of both finite and infinite-dimensional feedback operators. By applying this feedback operator and establishing necessary regularity results, we utilize the Banach fixed point theorem to demonstrate the stabilizability of the entire system. Furthermore, we also explore the stabilizability of the model problem around the zero steady state and validate our findings through numerical simulations using the finite element method for both zero and non-constant steady states.
\end{abstract}
\maketitle
\noindent \textbf{Keywords.} Generalize Burgers-Huxley equation with memory; Feedback stabilization; Distributed control; Non-constant steady state. \\
\noindent \textbf{MSC Classification (2020).} 93D15 $\cdot$ 93C20 $\cdot$ 35R09 $\cdot$ 35Q35 

\section{Introduction} 
The \emph{generalized Burgers-Huxley equation (GBHE)} serves as a prototype model for capturing the interplay between reaction mechanisms, convection effects, and diffusion transport. For more details about this model, we refer to the works \cite{WZLuGBHE, WazwazGBHE, SMAKSisc, MTM_Camwa,EDTS,MTMAnkit} and for stationary case see \cite{MTMAKRicGBHE}.

\subsection{Model problem}
Let $\Omega \subset \mathbb{R}^d,$ $d\in \{1,2,3\},$ be bounded domain with $C^2$ boundary $\Gamma.$ 
Consider the following generalized Burgers-Huxley equation with kernel 
\begin{equation} \label{eq:GBHE}
\left\{
    \begin{aligned}
& y_t - \eta \Delta y +  \alpha y^\delta \sum_{i=1}^d \frac{\partial y}{\partial x_i} - \kappa \int_0^t e^{-\lambda (t-s)}\Delta y(\cdot,s)ds  \\&\quad= \beta y(1-y^\delta)(y^\delta-\gamma)+ f \ \text{ in }\  \Omega\times (0,\infty), \\
& y=0  \ \text{ on }\  \Gamma\times [0,\infty), \\
        & y(x,0)=y_0 \ \text{ in }\ \Omega,
    \end{aligned}\right.
\end{equation}
where
\begin{itemize}
	\item $y=y(x,t)$ signifies the concentrations or density of a specific thing of interest, such as the population density of a chemical species or a biological environment;
	\item $\delta\in\N$ is a parameter which represents the degree of nonlinearity within the equation, subsequently controlling the magnitude of the response or feedback term in the equation; In fact it can take any value as given below: 
	\begin{align} \label{eqdef:delta}
		\delta=\begin{cases}
			\text{ any natural number, } & \text{ if } d=1, 2, \\
			1,2, & \text{ if } d=3;
		\end{cases}
	\end{align}
	\item $\alpha>0$ denotes the coefficient associated with the nonlinear advection term, capturing the energy of a transport phenomena or nonlinear wave propagation. It is known as  the advection coefficient;
	\item $\eta>0$, known as the viscosity coefficient, denotes the coefficient of diffusion term which helps us to understand  the spreading of the quantity $u$;
	\item $\beta>0$ represents the parameter associated to the reaction or feedback term, controlling the nonlinear reaction's or growth process's intensity;
	\item $\gamma\in(0,1)$ is another parameter associated with the reaction or feedback term, impacting the behavior of the system at different concentration levels;
	\item $ \lambda>0,$ the relaxation time (or decay rate), and $\kappa>0$ is a scaling factor. 
\end{itemize}
The model \eqref{eq:GBHE} with $\alpha=\delta=1, \kappa=0,$ and $\beta=0$ is the \emph{classical viscous Burgers} equation and the case  $\delta=1$ and $\kappa=0$ is known as the \emph{Burgers-Huxley equation}. The well-posedness of the problem \eqref{eq:GBHE}  is discussed in \cite[Theorem 2.11]{MTM_Camwa}, in fact, for the values of $\delta $ given in \eqref{eqdef:delta},  it is shown that for $y_0\in H^1_0(\Omega)$ and $f \in L^2(0,\infty; L^2(\Omega)),$  \eqref{eq:GBHE} has a unique strong solution 
\begin{align*}
    y \in L^2(0,\infty; H^2(\Omega)) \cap C([0,\infty); H^1_0(\Omega)) \cap L^{2(\delta+1)}(0,\infty; L^{6(\delta+1)}(\Omega)) \text{ with } y_t \in L^2(0,\infty; L^2(\Omega)).
\end{align*}
In the first part of this article, we provide an alternate and short proof of the existence and uniqueness of strong solution using semigroup approach and  Banach fixed point theorem. This method also avoids the tedious Faedo-Galerkin approximations which is usually find in the literature to obtain strong solutions to the problem \eqref{eq:GBHE} (\cite{SMAKSisc}).  First, we establish the existence of a local solution and then a global solution by using some energy estimates (see the Section \ref{sec:wellpsed}). Later, in this article, we study feedback stabilizability of GBHE with kernel around a non-constant steady state (see Section \ref{sec:stabaroundNonzero}) by using an interior control. Also, we briefly discuss the stabilizability around zero steady state (see Section \ref{sec:stabaroundzero}). Numerical studies have been carried out in Section \ref{sec:NS}.

\subsection{Literature survey}
Extensive literature exists on Burgers' equation and feedback stabilizability; interested readers can refer to \cite{BreKun20,BreKunPf19,SChowErve,JPR6-2017,JPR15SICON,BBSW15,BarbSICON12,JPRSheThev,BarbSICON11,BadESAIM9,AKR24,WMMTMNSEm} and related sources. However, for the current discussion, we highlight a few specific references. For instance, \cite{JPR,JPR2010} explored exponential stabilization of linearized Navier-Stokes equations (NSEs) around unstable stationary solutions using a boundary control. Authors in \cite{CRR-SICON,CMRR-JDE} investigated stabilization of one-dimensional compressible NSEs around constant steady states. Boundary feedback stabilization for two-dimensional fluid flow governed by NSEs with mixed boundary conditions is detailed in \cite{JPR15SICON}. Local stabilization of one-dimensional compressible NSEs around constant steady solutions was established in \cite{MRR_ADE}, while \cite{WKR} examined three-dimensional viscous Burgers equations with memory. Using spectral properties of the Oseen–Stokes operator, \cite{MuntNSEMem} studied NSEs with fading memory for stabilizability using boundary Dirichlet feedback controllers. They further explored integral equations through a coupled form and semigroup theory, yielding local stabilization results. For nonlinear parabolic systems, \cite{Badra-T} demonstrated stabilizability using finite-dimensional feedback controllers. Breiten and Kunisch in \cite{BreKun14} adopted a Riccati-based approach, echoed in \cite{Las-Tr-91,Lasiecka1,Lasieckabook}. Munteanu in \cite{Mont} addressed semilinear heat equations with fading memory, achieving stabilization with finite-dimensional boundary feedback controllers. In \cite{Thev-JPR}, authors stabilized two-dimensional Burgers equations using a nonlinear feedback controller derived from Hamilton-Jacobi-Bellman and algebraic Riccati equations. In one dimension, the authors in \cite{RaghuMTMKdvGBH} explored stabilization of the generalized Korteweg-de Vries-Burgers-Huxley equation in $L^2$ and $H^1$ using Lyapunov techniques and Neumann boundary nonlinear feedback control. Additionally, in \cite{RaghuMTMGBH}, they established exponential stability for the generalized Burgers-Huxley equation through appropriate boundary control strategies.

Local stabilization of viscous Burgers equation with memory is detailed in \cite{WKR}, focusing on reformulating equations in a coupled form. This approach is extended to NSEs with memory in \cite{WMMTMNSEm}, particularly regarding stabilizability around non-constant steady states. Challenges arise due to variable coefficients in the coupled equations' operators, complicating direct eigenvalue determination. 

\subsection{Methodology, contributions and novelty} \label{sec:contMethod}
This subsection focuses to discuss the methods we follow to establish the required results. By defining  $z=\int_0^t e^{-\lambda(t-s)}y(s) ds,$ one can re-write \eqref{eq:GBHE} as
\begin{equation} \label{eq:GBHE-u-coup}
\left\{
    \begin{aligned}
& y_t - \eta \Delta y +  \alpha y^\delta\sum_{i=1}^d \frac{\partial y}{\partial x_i} - \kappa\Delta z + \beta y(y^\delta-1)(y^\delta-\gamma) = f \ \text{ in }\  \Omega\times (0,\infty), \\
& z_t+\lambda z - y =0 \ \text{ in }\  \Omega\times (0,\infty),\\
& y=0 , \, \& \,z=0  \ \text{ on }\  \Gamma\times (0,\infty), \\
        & y(0)=y_0, \, \& \, z(0)=0 \ \text{ in }\ \Omega.
    \end{aligned}\right.
\end{equation}
We first show that the strong solution of a heat equation also belongs to $L^{2(\delta+1)}(0,\infty; L^{6(\delta+1)}(\Omega))$ (see Proposition \ref{pps:heatStrongL2L6}). Next, we consider the corresponding linear coupled equation
\begin{equation} \label{eq:GBHE-Lin-coup}
\left\{
    \begin{aligned}
& y_t - \eta \Delta y +\beta\gamma y - \kappa\Delta z = f  \ \text{ in }\   \Omega\times (0,\infty), \\
& z_t+\lambda z - y =0 \ \text{ in }\  \Omega\times (0,\infty),\\
& y=0 , \, \& \,z=0  \ \text{ on }\  \Gamma\times (0,\infty), \\
        & y(0)=y_0, \, \& \, z(0)=0 \ \text{ in }\ \Omega,
    \end{aligned}\right.
\end{equation}
which further can be written as 
\begin{equation} \label{eq:linOp}
	\Wf'(t)=\Af \Wf(t) + \begin{pmatrix} f(t) \\ 0 \end{pmatrix}, \text{ for all }t>0, \quad \Wf(0)= \begin{pmatrix} y_0 \\ 0 \end{pmatrix}=:\boldsymbol{w}_0,
\end{equation}
in $\Hf=L^2(\Omega)\times L^2(\Omega)$ with $\Wf =\begin{pmatrix} y \\ z \end{pmatrix},$ where the unbounded operator $\Af: D(\Af) \subset \Hf \rightarrow \Hf$ is defined as
\begin{equation} \label{eqdef:A}
	\begin{aligned}
		& \Af = \begin{pmatrix} \eta \Delta -\beta\gamma I & \kappa \Delta \\  I & -\lambda I    \end{pmatrix} \text{ with } D(\Af)=\left\lbrace \begin{pmatrix} w \\ v \end{pmatrix} \in \Hf \, |\, \eta w+\kappa v \in H^2(\Omega) \cap H^1_0(\Omega) \right\rbrace, \\
		& \text{ that is, } \Af\begin{pmatrix} w \\ v \end{pmatrix} = \begin{pmatrix} \Delta (\eta w+\kappa v) -\beta\gamma w \\ -\lambda v+w \end{pmatrix} \text{ for all } \begin{pmatrix} w \\ v \end{pmatrix} \in D(\Af),
	\end{aligned}
\end{equation}
with $I:L^2(\Omega)\to L^2(\Omega)$ being the identity operator. Next, we establish a regularity result (see Step 1 of proof of Theorem \ref{th:localStrongsolGBHE}) for the system \eqref{eq:GBHE-Lin-coup} using semigroup method and regularity for a heat equation.  A utilization of Banach fixed point theorem leads to the existence of a local strong solution of the the coupled non-linear system \eqref{eq:GBHE-u-coup} which further extended to a global strong unique solution with a help of energy estimates (see Theorem \ref{th:localStrongsolGBHE} and its proof).

In the next part of this article, we focus to study feedback stabilization of GBHE with memory
\begin{equation} \label{eq:GBHENC-intro}
	\left\{
	\begin{aligned}
		& y_t - \eta \Delta y +  \alpha y^\delta \sum_{i=1}^d \frac{\partial y}{\partial x_i} - \kappa \int_0^t e^{-\lambda (t-s)}\Delta y(\cdot,s)ds  \\&\quad= \beta y(1-y^\delta)(y^\delta-\gamma) +f_\infty + u\chi_{\mathcal{O}} \ \text{ in }\  \Omega\times (0,\infty), \\
		& y=0  \ \text{ on }\  \Gamma\times (0,\infty), \\
		& y(x,0)=y_0(x) \ \text{ in }\ \Omega,
	\end{aligned}\right.
\end{equation}
where $f_\infty \in L^2(\Omega)$ is a given stationary force term, 
around zero and a non-constant steady state $y_\infty,$ the solution of
\begin{equation} \label{eq:GBHE-ST-intro}
	\left\{
	\begin{aligned}
		&  - \eta \Delta y_\infty +  \alpha y_\infty^\delta \sum_{i=1}^d \frac{\partial y_\infty}{\partial x_i}   + \beta y_\infty(y_\infty^\delta-1)(y_\infty^\delta-\gamma) - \frac{\kappa}{\lambda}\Delta y_\infty =f_\infty  \ \text{ in }\  \Omega, \\
		& y_\infty=0  \ \text{ on }\  \Gamma.
	\end{aligned}\right.
\end{equation}
For any $f_{\infty}\in L^2(\Omega)$, the existence of a strong solution $y_{\infty}\in H^2(\Omega)\cap H_0^1(\Omega)$  to the problem \eqref{eq:GBHE-ST-intro} is obtained in \cite[Theorem 2.1]{MTMAKRicGBHE}.  Here, we skip the details of the stabilizability around zero steady state but we focus on the methodology of the stabilizability around non-constant steady state. Note that stabilizability of $y$ around the steady state $y_\infty$ means to show the stabilizability of showing $y-y_\infty$ around zero. Hence, it is meaningful to introduce $w:=y-y_\infty$ and discuss the stabilizability of $w$ around zero, where the equation satisfied by $w$ is
\begin{equation}  \label{eq:GBHE-Lin-w-y_inft-intro}
	\left\{
	\begin{aligned}
		&  w_t -\eta \Delta w  + \alpha \left( (w+y_\infty)^\delta \nabla (w+y_\infty) \cdot \boldsymbol{1} - y_\infty^\delta \nabla y_\infty\cdot \boldsymbol{1} \right)  - \kappa \int_0^t e^{-\lambda(t-s)}\Delta w(\cdot,s) ds \\
		& \  + \beta \left( (w+y_\infty) ((w+y_\infty)^\delta -1) ((w+y_\infty)^\delta -\gamma) - y_\infty (y_\infty^\delta -1) (y_\infty^\delta -\gamma)  \right)\\&\ + \frac{\kappa}{\lambda}e^{-\lambda t}\Delta y_\infty =u \chi_{\mathcal{O}}  \ \text{ in }\  \, \Omega\times(0,\infty),\\
		&   w(x,t)=0 \ \text{ on }\   \Gamma\times (0,\infty),\\
		&w(x,0)=y_0(x)-y_\infty(x)  \ \text{ in }\  \, \Omega.
	\end{aligned}
	\right.
\end{equation}
Here, $\boldsymbol{1}$ denotes the vector $(1,1,\ldots,1)^{\top} \in \mathbb{R}^d.$ Our first aim is to consider a principal system (see \cite{WMMTMNSEm})
\begin{equation}  
	\left\{
	\begin{aligned}
		&  w_t -\eta \Delta w +\beta\gamma w  - \kappa \int_0^t e^{-\lambda(t-s)}\Delta w(\cdot,s) ds =u \chi_{\mathcal{O}} \ \text{ in }\   \Omega\times (0,\infty),\\
		&   w(x,t)=0\  \text{ on } \  \Gamma\times (0,\infty),\\
		&w(x,0)=y_0(x)-y_\infty=:w_0 \ \text{ in} \  \Omega,
	\end{aligned}
	\right.
\end{equation}
and study the feedback stabilizability by writing it in a coupled equation of the form 
\begin{align*}
	X'(t)=\Af X(t) + \Bf u(t) \text{ for all }t>0, \quad X(0)=X_0,
\end{align*}
in $\Hf,$ with $X(t)=\begin{pmatrix} w(\cdot,t) \\ v(\cdot,t) \end{pmatrix},$ where $\Af$ is the same as in  \eqref{eqdef:A}, $\Bf=\begin{pmatrix} u \chi_{\mathcal{O}} \\ 0 \end{pmatrix}$ is the control operator, $X_0=\begin{pmatrix} y_0-y_\infty \\ 0 \end{pmatrix},$ and $v(\cdot,t)=\int_0^t e^{-\lambda(t-s)}w(\cdot,s) ds.$
The stabilizability of such system is available in literature (see \cite{WKR}) and achieved using semigroup theory, leveraging the spectral properties of 
$\Af$ and verifying the Hautus condition. This is  notably, since the operator $\Af$ is the same as in \cite{WKR} up to a bounded perturbation, the stabilizability result closely follows from \cite{WKR} (exactly same if we take $\kappa=1$ and $\beta=0$ (or $\gamma=0$)) and therefore the properties of $\Af$ are analogous with the one in \cite{WKR}. Here, we simply state this result for the principal system. Subsequently, we establish a regularity result for the full nonlinear system \eqref{eq:GBHE-Lin-w-y_inft-intro} and develop a local stabilization result using the feedback operator obtained for the case of principal system.
    
We use finite element method to verify our results (both zero and non-constant steady states) by considering some numerical examples in respective subsections of this article.

This article introduces a novel approach by offering a concise alternative proof for the existence and uniqueness of strong solutions to GBHE with memory. The techniques and proofs presented here are extendable to various other evolution equations with memory, establishing their well-posedness. Furthermore, a pioneering aspect of this study involves analyzing the stabilizability of nonlinear parabolic equations with memory around non-constant steady states, which has not been extensively explored in the existing literature. The inclusion of nonlinearities in exploring feedback stabilizability in both two and three dimensions adds further intrigue, paving the way for future investigations into the stabilizability of diverse nonlinear evolution equations, whether with or without memory, around non-constant steady states.

\subsection{Preliminaries and notations} 
We provide the function spaces needed to obtain the required results and some basic inequalities in this section. Let $L^p(\Omega)$ denote the space of equivalence classes of Lebesgue measurable functions $f: \Omega \to \mathbb{R}$ such that $ \int_{\Omega} |f(x)|^p \, dx < \infty.$
Two measurable functions are equivalent if they are equal almost everywhere. The $L^p$-norm of $f \in L^p(\Omega)$ is defined as
$ \|f\|_{L^p(\Omega)} := \left( \int_{\Omega} |f(x)|^p \, dx \right)^{1/p}. $ For $p = 2$, $L^2(\Omega)$ forms a Hilbert space where the inner product $(\cdot, \cdot)$ is given by
$ (f,g) = \int_{\Omega} f(x) g(x) \, dx, $ and the corresponding norm is denoted by $\|\cdot\|$. Moreover, $H_0^1(\Omega)$, also denoted by  $W^{1,2}_0(\Omega)$, represents the Sobolev space defined as the set of equivalence classes of Lebesgue measurable functions $f \in L^2(\Omega)$ with weak derivatives $\nabla f \in L^2(\Omega)$ and zero trace. The norm in $H_0^1(\Omega)$ is given by
$ \|f\|_{H_0^1(\Omega)} := \left( \int_{\Omega} |\nabla f(x)|^2 \, dx \right)^{1/2}, $
utilizing the Poincar\'e inequality. We define $H^{-1}(\Omega) := (H_0^1(\Omega))'$, the dual of $H_0^1(\Omega)$, with norm
\[ \|g\|_{H^{-1}(\Omega)} := \sup \left\{ \langle g, f \rangle : f \in H_0^1(\Omega), \|f\|_{H_0^1(\Omega)} \leq 1 \right\}, \]
where $\langle \cdot, \cdot \rangle$ denotes the duality pairing between $H_0^1(\Omega)$ and $H^{-1}(\Omega)$. 
We denote second-order Hilbertian Sobolev spaces by $H^2(\Omega)$, and in general, Sobolev spaces by $W^{m,p}(\Omega)$, where $m \in \mathbb{N}$ and $p \in [1, \infty)$. 
We recall some useful inequalities that is used in the paper frequently. 

\noindent \textbf{Young's inequality.} Let $a,b$ be any non-negative real numbers. Then for any $\varepsilon>0,$ the following inequalities hold:
\begin{align} \label{eqPR-YoungIneq}
	ab \le \frac{\varepsilon a^2}{2} + \frac{b^2}{2\varepsilon} \text{ and } ab \le \frac{a^p}{p}+\frac{b^q}{q},
\end{align}
for any $p,q>1$ such that $\frac{1}{p}+\frac{1}{q}=1.$
\begin{Proposition}[Generalized H\"older's inequality] \label{ppsPR-GenHoldIneq}
	Let $f\in L^p(\Omega),$ $g\in L^q(\Omega),$ and $h\in L^r(\Omega),$ where $1\le p, q ,r\leq \infty$ are such that $\frac{1}{p}+\frac{1}{q}+\frac{1}{r}=1.$ Then $fgh\in L^1(\Omega)$ and 
	\begin{align*}
		\|fgh\|_{L^1(\Omega)} \le \|f\|_{L^p(\Omega)} \|g\|_{L^q(\Omega)} \|h\|_{L^r(\Omega)}.
	\end{align*} 
\end{Proposition}
\begin{Lemma}[Agmons' inequality {\n\cite{Agmon10}}] \label{lemVB:AgmonIE}
	Let $\Omega$ be an open bounded domain in $\mathbb{R}^d,$  and let $f \in H^{s_2}(\Omega).$ Let $s_1,s_2$ be such that $s_1<\frac{d}{2}<s_2.$ If for $0<\theta<1,$ $\frac{d}{2}=\theta s_1+(1-\theta)s_2,$ then there exists a positive constant $C_a=C_a(\Omega)$ such that
	\begin{equation} \label{eqVB-AgmonIE}
		\|f\|_{L^\infty(\Omega)} \le C_a \|f\|_{H^{s_1}(\Omega)}^{\theta} \|f\|_{H^{s_2}(\Omega)}^{1-\theta}.
	\end{equation}
\end{Lemma}
In the following lemma, we recall Sobolev embedding \cite[Theorem 2.4.4]{Kes} in our context.
\begin{Lemma}[Sobolev embedding] \label{lemPR:SobEmb}
	Let $\Omega$ be an open bounded domain in $\mathbb{R}^d$ of class $C^1$  with $d\in \mathbb{N}.$ Then, we have the following continuous inclusion with constant $C_s$:
	\begin{itemize}
		\item[$(a)$] $H^1(\Omega)\hookrightarrow L^p(\Omega)$ for $d>2,$ where $p=\frac{2d}{d-2},$
		\item[$(b)$] $H^1(\Omega) \hookrightarrow L^q(\Omega)$ for all $d=2\le q<\infty.$ 
	\end{itemize}
\end{Lemma}
In the next theorem, we recall the important inequality, known as the Gagliardo-Nirenberg inequality in the case of bounded domains with smooth boundary.
\begin{Theorem}[Gagliardo-Nirenberg inequality {\normalfont\cite{Nirenberg1959}}] \label{thm:G-NinBdd}
	Let \( \Omega \subset \mathbb{R}^d \) be a bounded domain with a smooth boundary. For any \( u \in W^{m,q}(\Omega) \), and for any integers \( j \) and \( m \) satisfying \( 0 \leq j < m \), the following inequality holds:
	\[
	\| D^j u \|_{L^p(\Omega)} \leq C \|  u \|_{W^{m,q}(\Omega)}^\theta \| u \|_{L^r(\Omega)}^{1-\theta} ,
	\]
	where \( D^j u \) denotes the \( j \)-th order weak derivative of \( u \), and:
	\[
	\frac{1}{p} = \frac{j}{d}+\theta \left( \frac{1}{q} - \frac{m}{d} \right) +  \frac{(1-\theta)}{r},
	\]
	for some constant \( C \) that depends on the domain \( \Omega \) but not on \( u \), where \( \theta \in [0,1] \).
\end{Theorem}

\begin{Lemma}[General Gronwall inequality \cite{Canon99}] \label{lem:Gronwall} 
	Let $f,g, h$ and $y$ be four locally integrable non-negative functions on $[t_0,\infty)$ such that
	\begin{align*}
		y(t)+\int_{t_0}^t f(s)ds \le C+ \int_{t_0}^t h(s)ds + \int_{t_0}^t g(s)y(s)ds\  \text{ for all }t\ge t_0,
	\end{align*}
	where $C\ge 0$ is any constant. Then 
	\begin{align*}
		y(t)+\int_{t_0}^t f(s)ds \le \left(C+\int_{t_0}^t h(s)ds\right) \exp\left( \int_{t_0}^t g(s)ds\right) \text{ for all }t\ge t_0.
	\end{align*}
\end{Lemma}
Now, we state a version of nonlinear generalization of Gronwall’s inequality that will be useful in our later analysis.
\begin{Theorem}{(A nonlinear generalization of Gronwall’s inequality \cite[Theorem 21]{DragSS}).}\label{lem:NonlinGronwall}
Let $\zeta(t)$ be a non-negative function that satisfies the integral inequality
\begin{align*}
	\zeta(t)\le c+ \int_{t_0}^t \left(  a(s)\zeta(s) + b(s)\zeta^{\varepsilon}(s) \right) ds, c\ge 0, \, 0\le \varepsilon<1,
\end{align*}
where $a$ and $b$ are locally integrable non-negative functions on $[t_0,\infty).$ Then, the following inequality holds: 
\begin{align*}
	\zeta(t) \le \left\lbrace c^{1-\varepsilon} \exp{\left[(1-\varepsilon)\int_{t_0}^ta(s)ds\right]} + (1-\varepsilon)\int_{t_0}^t b(s) \exp{\left[(1-\varepsilon)\int_{s}^t a(r)dr \right]} ds\right\rbrace^{\frac{1}{1-\varepsilon}}.
\end{align*}
\end{Theorem}

\subsection{Organization} The remainder of this article is organized as follows. In Section \ref{sec:PrincOp}, we revisit the semigroup framework and the spectral properties of the operator $\Af$. Section \ref{sec:wellpsed} presents an alternative proof for the existence and uniqueness of strong solutions to the problem  \eqref{eq:GBHE}. In Section \ref{sec:stabaroundzero}, we investigate the feedback stabilizability of the model problem around the zero steady state. The first subsection addresses the linear case, while the following subsection extends the analysis to the full nonlinear system.  Section \ref{sec:stabaroundNonzero} explores the stabilizability around a non-constant steady state. The article wraps up by providing some numerical simulations to verify the findings for stabilizability around both zero and non-constant steady-state cases in Section \ref{sec:NS}.

\section{Principal operator} \label{sec:PrincOp}
Following \cite{WKR}, we mention that the operator $(\Af, D(\Af))$ generates an analytic semigroup on $\Hf$ and we also discuss the spectral properties of $(\Af, D(\Af))$.

\begin{Theorem}{\n \cite[Proposition 3.1 and Theorem 3.2]{WKR}} \label{th:welPosedLinOp}
	\begin{enumerate}
	\item [(a)]  The operator $(\Af, D(\Af))$ in \eqref{eqdef:A} is densely defined, closed,
	and generates an analytic semigroup $\{e^{t\Af}\}_{t\ge 0}$ on $\Hf.$ \\
	\item [(b)] For any $\Wf_0\in \Hf$ and $ \boldsymbol{f}\in L^2(0,\infty; \Hf),$ the system 
	\begin{align} \label{eq:opform-f}
		\Wf'(t)=\Af \Wf(t) + \boldsymbol{f}(t) \text{ for all }t>0, \quad \Wf(0)=\Wf_0,
	\end{align}
	admits a unique strong solution $\Wf \in L^2(0,\infty; \Hf).$ Moreover, $\Wf$ belongs to $C([0,\infty); \Hf)$ with the representation 
	\begin{align*}
		\Wf(t)=e^{t\Af}\Wf_0 +\int_0^t e^{(t-s)\Af} \boldsymbol{f}(s) ds \text{ for all }t>0.
	\end{align*}
	\end{enumerate}
\end{Theorem}
\begin{proof}
	Note that 
	\begin{align*}
		\Af=A_1+A_2, \text{ where } A_1= \begin{pmatrix} \eta \Delta  & \kappa \Delta \\  I & -\lambda I    \end{pmatrix} \text{ with } D(A_1)=D(\Af),
	\end{align*}
	and $A_2= \begin{pmatrix} -\beta\gamma I & 0 \\  0 & 0 \end{pmatrix} $ is bounded operator on $\Hf$. From, \cite[Proposition 3.1 and Theorem 3.2]{WKR}, $(A_1, D(A_1))$ generates an analytic semigroup on $\Hf.$ Now, $(\Af, D(\Af))$ being bounded perturbation of $(A_1,D(A_1)),$ the result \cite[Theorem 12.37]{RROG} yields that $(\Af, D(\Af))$ generates an analytic semigroup $\{ e^{t\Af}\}_{t\ge 0}$ on $\Hf.$ The
	part (b) follows from \cite[Part II,Ch. 1, Proposition 3.1]{BDDM} using part (a).
\end{proof}
Now, we discuss the spectral properties of $(\Af, D(\Af))$ in $\Hf.$ To do that, we recall the eigenvalue problem (see \cite[Theorem 1, Section 6.5]{Eva})
\begin{equation}\label{eq:eigvalLapl}
	\begin{aligned}
		 -\Delta \phi &= \Lambda \phi \ \text{ in }\  \Omega, \\
		 \phi&=0 \ \text{ on }\  \Gamma,
	\end{aligned}
\end{equation}
that has an infinite sequence $\{ \Lambda_k\}_{k\in \mathbb{N}}$ of eigenvalues with 
\begin{align*}
	0<\Lambda_1\le \Lambda_2 \le \cdots \le \Lambda_k \to \infty,
\end{align*}
eigenfunction $\phi_k \in C^\infty(\Omega)\cap H^1_0(\Omega)$ (\cite[Corollary 6.14]{RobiIDDS}) corresponding to $\Lambda_k,$ for each $k\in \mathbb{N},$ and $\left\lbrace \phi_k \right\rbrace_{k\in \mathbb{N}}$ forms an orthonormal basis in $L^2(\Omega).$ Therefore, $\left\lbrace \begin{pmatrix} \phi_k \\ 0 \end{pmatrix},  \begin{pmatrix} 0 \\ \phi_k \end{pmatrix} \, |\, k \in \mathbb{N} \right\rbrace$ is an orthonormal basis for $\Hf.$ 
In the next result, we discuss the eigenvalues of $\Af$ and they also satisfy the similar properties as in \cite[Proposition 3.3]{WKR}.

\begin{Theorem}[Spectral analysis]\label{th:specAna}
	Let $(\Af, D(\Af))$ be defined as in \eqref{eqdef:A}. The eigenvalues of $\Af$  consist of two sequences $\mu_k^+$ and $\mu_k^-,$ $k\in\{1,2,\ldots\}$  given by
	\begin{align}
		\mu_k^\pm = \frac{-(\beta\gamma+\lambda+\eta \Lambda_k) \pm \sqrt{(\beta\gamma+\lambda+\eta \Lambda_k)^2 - 4 (\beta\gamma\lambda+ (\eta \lambda +\kappa)\Lambda_k)}}{2},
	\end{align}
	where $\Lambda_k,$ $k\in\{1,2,\ldots\}$ are eigenvalues of $-\Delta$ as discussed in \eqref{eq:eigvalLapl}. Furthermore, we have the following properties: 
	\begin{itemize}
		\item[(a)] There are only finitely many complex eigenvalues of $\Af.$ In fact, $\mu_k^\pm,$ for all $k ,$ for which $\Lambda_k$ satisfies 
		\begin{align*}
			\frac{(\eta \lambda +2 \kappa -\beta\gamma\eta) - 2 \sqrt{\kappa^2+\kappa \eta \lambda-\beta\gamma\eta\kappa}}{\eta^2} <\Lambda_k < \frac{(\eta \lambda +2 \kappa-\beta\gamma\eta) + 2 \sqrt{\kappa^2+\kappa \eta \lambda-\beta\gamma\eta\kappa}}{\eta^2},
		\end{align*}
		are complex eigenvalues of $\Af.$
		\item[(b)] The sequence $\mu_k^+$ converges to $-\nu_0$ as $k\rightarrow \infty$ while the other sequence $\mu_k^-$ behaves like $-\eta\Lambda_k$ and goes to $-\infty$ as $k\to \infty,$ where
		\begin{align} \label{eq:nu0}
			\nu_0:= \left(\lambda+\frac{\kappa}{\eta}\right).  
		\end{align}
		\item[(c)] All the eigenvalues have negative real part. In particular, there exists $N_1,N_2\in \mathbb{N}$ such that 
		\begin{align*}
			& \text{ for all }k< N_1, \quad  - \nu_0 < \Re(\mu_k^+)<0 ,    \quad \text{ for all }k\ge N_1, \Re(\mu_k^+)< -\nu_0, \\
			& \text{ for all }k< N_2, \quad  - \nu_0 < \Re(\mu_k^-)<0 ,    \quad \text{ for all }k\ge N_2, \Re(\mu_k^-)< - \nu_0.
		\end{align*}
	\end{itemize}
\end{Theorem}

Next, we calculate the adjoint operator $(\Af^*, D(\Af^*))$ (see \cite[Section A.2]{WKRPCE} for a similar calculation) of $(\Af, D(\Af))$ as
\begin{align} \label{eqdef:A*}
	\Af^* \begin{pmatrix} \phi \\ \psi \end{pmatrix} = \begin{pmatrix} \eta\Delta \phi -\beta\gamma \phi +\psi\\ \kappa \Delta \phi-\lambda \psi \end{pmatrix} \text{ for all } \begin{pmatrix} \phi \\ \psi \end{pmatrix} \in D(\Af^*):=\left\lbrace \begin{pmatrix} \phi \\ \psi \end{pmatrix} \in \Hf \, |\, \phi \in H^2(\Omega)\cap H^1_0(\Omega) \right\rbrace.  
\end{align}
One can choose eigenfunctions $\xi_k^\pm$ of $\Af$ corresponding to eigenvalues $\mu_k^\pm$ as 
\begin{align} \label{eq:eigfun-A}
	\xi_k^\pm = \begin{pmatrix} 1 \\ \frac{1}{\mu_k^\pm+\lambda} \end{pmatrix}\phi_k \text{ for all }k\in \mathbb{N}.
\end{align}
It can be easily seen that the set of eigenvalues of $\Af^*$ is $\left\lbrace \overline{\mu_k^+}, \overline{\mu_k^-}, \, |\, k \in \mathbb{N} \right\rbrace$. For all $k\in \mathbb{N},$ we calculate eigenfunctions $\xi_k^{*+}$ and $\xi_k^{*-}$ of $\Af^*$ corresponding to the eigenvalues $\overline{\mu_k^+}$ and $ \overline{\mu_k^-},$ respectively as
\begin{align} \label{eq:eigfun-A*}
	\xi_k^{*+} =    \frac{\left(\lambda+ \overline{\mu_k^+}\right)^2}{\left(\lambda+ \overline{\mu_k^+}\right)^2-\eta \Lambda_k}  \begin{pmatrix} 1 \\ \frac{-\eta \Lambda_k}{\lambda+ \overline{\mu_k^+}} \end{pmatrix} \phi_k \text{ and } \xi_k^{*-} =    \frac{\left(\lambda+ \overline{\mu_k^-}\right)^2}{\left(\lambda+ \overline{\mu_k^-}\right)^2-\kappa \Lambda_k}  \begin{pmatrix} 1 \\ \frac{-\kappa \Lambda_k}{\lambda+ \overline{\mu_k^-}} \end{pmatrix} \phi_k.
\end{align}
Thus the analysis done in \cite[Section 3.2 -3.3 ]{WKR}  are valid here and consequently we obtain the following result.
\begin{Theorem}
	The family of eigenfunctions of $\Af,$ $\{ \xi_k^+, \xi_k^- \, |\, k \in \mathbb{N} \},$ given in \eqref{eq:eigfun-A}, forms a Riesz basis in $\Hf.$ The same is also true for the family of eigenfunctions in \eqref{eq:eigfun-A*} of $\Af^*,$ $\{ \xi_k^{*+}, \xi_k^{*-} \, |\, k \in \mathbb{N} \}.$ Furthermore, the spectrum of $\Af,$ denoted by $\sigma(\Af),$ is closure to the set of eigenvalues of $\Af,$ that is,
	\begin{align}
		\sigma(\Af):=\text{ closure of } \left\lbrace \mu_k^+, \mu_k^- \, |\, k\in \mathbb{N} \right\rbrace \ \text{ in }\  \mathbb{C} = \left\lbrace \mu_k^+, \mu_k^- \, |\, k\in \mathbb{N} \right\rbrace \cup \{-\nu_0\}.
	\end{align}
\end{Theorem}

\section{Strong solution} \label{sec:wellpsed}  
In this section, we provide an alternate proof of the existence and uniqueness of strong solutions of the problem \eqref{eq:GBHE}. To do that, let us first consider a heat equation and provide a result on the regularity of the same. 
\begin{Proposition}\label{pps:heatStrongL2L6}
Let $\delta$ be as given in \eqref{eqdef:delta}, $w_0\in H^1_0(\Omega)$ and $f \in L^2(0,\infty; L^2(\Omega)).$ Then there exists a unique strong solution $$w \in L^2(0,\infty; H^2(\Omega)) \cap L^\infty(0,\infty; H^1_0(\Omega)) \cap L^{2(\delta+1)}(0,\infty; L^{6(\delta+1)}(\Omega)) \text{ with } w_t \in L^2(0,\infty; L^2(\Omega))$$ of the problem 
\begin{equation*}
\left\lbrace
\begin{aligned}
    & w_t(x,t) -\eta \Delta w(x,t)=f(x,t) \ \text{ in }\  \Omega \times (0,\infty), \\
    & w(x,t)=0 \ \text{ on }\ \Gamma, \\
    & w(x,0)=w_0(x) \ \text{ in }\ \Omega,
\end{aligned}\right.
\end{equation*}
such that for all $\psi \in L^2(\Omega)$, the following equality is satisfied: 
\begin{align}
	(w_t(t),\psi(t))-\eta(\Delta w(t),\psi(t))=(f(t),\psi(t)), \ \text{ for a.e. } \ t\geq 0. 
\end{align}
Furthermore, the following estimate holds
\begin{align*}
	\|w\|_{L^2(0,\infty; H^2(\Omega))} + \|w\|_{L^\infty(0,\infty; H^1_0(\Omega))} & + \|w\|_{H^1(0,\infty; L^2(\Omega))}+\|w\|_{L^{2(\delta+1)}(0,\infty; L^{6(\delta+1)}(\Omega))} \\
	& \le C\left( \|w_0\|_{H^1_0(\Omega)}  +\|f\|_{L^2(0,\infty;L^2(\Omega))}  \right),
\end{align*}
for some positive constant $C.$
\end{Proposition}
An idea of the proof is provided in Appendix \ref{app:heat}.

Next, we aim to demonstrate the existence and uniqueness of strong solutions of GBHE with memory:
\begin{Theorem} \label{th:localStrongsolGBHE}
Let $\delta$ be as given in \eqref{eqdef:delta} and $T>0$ be any fixed number. For any $y_0\in H^1_0(\Omega)$ and $f\in L^2(0, T; L^2(\Omega)),$ the following generalized Burgers-Huxley equation with memory
\begin{equation} \label{eq:GBHE-f}
\left\{
    \begin{aligned}
& y_t - \eta \Delta y +  \alpha y\sum_{i=1}^d \frac{\partial y}{\partial x_i} - \kappa \int_0^t e^{-\lambda (t-s)}\Delta y(s)ds + \beta y(y^\delta-1)(y^\delta-\gamma) = f \ \text{ in }\  \Omega\times (0,T), \\
& y=0  \ \text{ on }\  \Gamma\times (0,T), \\
        & y(0)=y_0 \ \text{ in }\ \Omega.
    \end{aligned}\right.
\end{equation}
admits a unique strong solution with   $$y\in L^2(0,T; H^2(\Omega))\cap L^\infty(0,T; H^1_0(\Omega)) \cap L^{2(\delta+1)}(0,T; L^{6(\delta+1)}(\Omega))$$ with $y_t \in L^2(0,T; L^2(\Omega) )$ satisfying 
\begin{align}\label{en-est}
	\|y\|_{L^\infty(0,T; H^1_0(\Omega))}  + \|y\|_{L^2(0,T; H^2(\Omega))} &+ \|y\|_{H^1(0,T; L^2(\Omega))}  + \|y\|_{L^{2(\delta+1)}(0,T; L^{6(\delta+1)}(\Omega))} \\
	&\le C \left( \|y_0\|_{H^1_0(\Omega)}  +\|f\|_{L^2(0,T; L^2(\Omega))} \right), 
\end{align}
for some $C>0.$
\end{Theorem}
In the above theorem, we aim to obtain a global existence result but in the case of infinite time horizon we get a local existence result in the sense that we need to take a smallness assumption on the intial data and the forcing term. In the first two steps of  the prove of the Theorem\ref{th:localStrongsolGBHE} (see. page 11) we consider infinite time horizone to obtain local existence result, however, in the step 3 of the proof we consider finite time to obtain a global existence result.
To prove the result, we first establish the well-posedness of the corresponding linear system and then apply the Banach fixed-point theorem. This will ensure us a local existence result and that further can be extended using an energy estimate. To apply the Banach fixed-point theorem, we need the following two crucial Propositions. 

Let us define the non-linear terms as 
\begin{align} \label{eq:NL-h-GBHE}
    h(\psi):= \alpha \psi^\delta \nabla \psi\cdot \boldsymbol{1},
\end{align}
and 
\begin{align} \label{eq:NL-g-GBHE}
    g(\psi):=\beta \psi( \psi^\delta-1)( \psi^\delta-\gamma)-\beta \gamma \psi = \beta \psi^{2\delta+1}-\beta(1+\gamma)\psi^{\delta+1}.
\end{align}
We also define the space 
\begin{align} \label{eq:D}
   D:= \left\lbrace \psi\in L^2(0,\infty; H^2(\Omega)) \cap L^\infty(0,\infty; H^1_0(\Omega)) \cap L^{2(\delta+1)}(0,\infty; L^{6(\delta+1)}(\Omega))\cap H^1(0,\infty; L^2(\Omega))  \right\rbrace
\end{align}
equipped with the norm
\begin{align}
    \|\psi\|_D^2:= \|\psi\|_{L^2(0,\infty; H^2(\Omega))}^2 + \|\psi\|_{L^\infty(0,\infty; H^1_0(\Omega))}^{2} + \|\psi\|_{L^{2(\delta+1)}(0,\infty; L^{6(\delta+1)}(\Omega))}^{2} + \|\psi\|_{H^1(0,\infty; L^2(\Omega)) }^2.
\end{align}

\begin{Proposition} \label{pps:GBHE-selfmap-I}
For any $\psi \in D,$ the functions $h$ and $g$ defined in \eqref{eq:NL-h-GBHE} and \eqref{eq:NL-g-GBHE}, respectively, satisfy the following:
\begin{itemize}
\item[(a)] $\|h(\psi)\|_{L^2(0,\infty; L^2(\Omega))} \le M_1  \|\psi\|_D^{\delta+1},$
\item[(b)] $\|g(\psi)\|_{L^2(0,\infty; L^2(\Omega))} \le M_1 \left(\|\psi\|_D^{\delta+1} + \|\psi\|_D^{2\delta+1}\right) ,$
\end{itemize}
for some $M_1>0.$
\end{Proposition}

\begin{proof}
(a) Note that in the case of two-dimensions, we have $H^1_0(\Omega) \hookrightarrow L^{3\delta}(\Omega)$ for any $\frac{2}{3}\le \delta<\infty,$ by using Sobolev embedding.  However, in the case of three dimension it is true for $1\le \delta\le 2.$ We start with applying generalized H\"{o}lder's inequality and Sobolev embedding to obtain 
\begin{align*}
 \|h(\psi)\|_{L^2(0,\infty; L^2(\Omega))}^2 & = \int_0^\infty \|  \alpha  \psi^\delta(t) \nabla \psi(t)\cdot \boldsymbol{1} \|^2 dt  \\
 & \le \alpha^2 \int_0^\infty \|\psi^\delta(t)\|_{L^3(\Omega)}^2 \|\nabla \psi(t)\|_{L^6(\Omega)}^2 dt \\
 & \le C_s \alpha^2 \int_0^\infty \left( \int_\Omega |\psi(x,t)|^{3\delta} dx  \right)^{\frac{2}{3\delta}\delta} \|\psi(t)\|_{H^2(\Omega)}^2 dt \\
 & \le C_s\alpha^2 \|\psi\|_{L^\infty(0,\infty; L^{3\delta}(\Omega))}^{2\delta} \|\psi\|_{L^2(0,\infty; H^2(\Omega))}^2.
\end{align*}

\medskip 
\noindent (b) Note that $g(\psi)= \beta \psi^{2\delta+1}-\beta (1+\gamma)\psi^{\delta+1}.$ Let us first try to estimate the first term, that is,  $\psi^{2\delta+1},$ by using H\"{o}lder's inequality:
\begin{align*}
 \|\psi^{2\delta+1}\|_{L^2(0,\infty; L^2(\Omega))}^2 & = \int_0^\infty  \|\psi^{\delta+1}(t) \psi^\delta(t)\|^2 dt  \\
 & \le \int_0^\infty \|\psi^{\delta+1}(t)\|_{L^6(\Omega)}^2 \|\psi^\delta(t)\|_{L^3(\Omega)}^2 dt \\
 & \le \|\psi\|_{L^\infty(0,\infty; L^{3\delta}(\Omega))}^{2\delta}  \|\psi\|_{L^{2(\delta+1)}(0,\infty; L^{6(\delta+1)}.(\Omega))}^{2(\delta+1)}.
\end{align*}
We estimate the second term in a similar way:
\begin{align*}
	\|\psi^{\delta+1}\|_{L^2(0,\infty; L^2(\Omega))}^2 & = \int_0^\infty  \|\psi(t) \psi^\delta(t)\|^2 dt  \\
	& \le \int_0^\infty \|\psi(t)\|_{L^6(\Omega)}^2 \|\psi^\delta(t)\|_{L^3(\Omega)}^2 dt \\
	& \le \|\psi\|_{L^\infty(0,\infty; L^{3\delta}(\Omega))}^{2\delta}  \|\psi\|_{L^{2}(0,\infty; L^{6}(\Omega))}^{2}.
\end{align*}
Combing the above estimates, we conclude the proof.
\end{proof}

\begin{Proposition}\label{pps:GBHE-contrac-I}
For any $\psi_1, \psi_2 \in D,$ the functions $h$ and $g$ defined in \eqref{eq:NL-h-GBHE} and \eqref{eq:NL-g-GBHE}, respectively, satisfy the following estimates:
\begin{itemize}
\item[(a)] $\|h(\psi_1) - h(\psi_2)\|_{L^2(0,\infty; L^2(\Omega))} \le M_2 \Big(\|\psi_1\|_D^{\delta}+\|\psi_2\|_D^{\delta}+\|\psi_2\|_D^{\delta-1}\|\psi_1\|_D \Big) \|\psi_1 - \psi_2\|_D,$
\item[(b)] $\|g(\psi_1) - g(\psi_2)\|_{L^2(0,\infty; L^2(\Omega))} \le M_2 \left(\|\psi_1\|_D^{2\delta}+\|\psi_2\|_D^{2\delta}+\|\psi_1\|_D^\delta+\|\psi_2\|_D^\delta\right)\|\psi_1 - \psi_2\|_D,$
\end{itemize}
for some $M_2>0.$
\end{Proposition}

\begin{proof}
(a) First, we write 
{\small\begin{align}\label{eqn-diff}
 h(\psi_1) - h(\psi_2) =    \alpha  \psi_1^\delta \nabla \psi_1 \cdot \boldsymbol{1} -  \alpha  \psi_2^\delta  \nabla \psi_2 \cdot \boldsymbol{1} = \alpha \left( (\psi_1^\delta - \psi_2^\delta) \nabla \psi_1\cdot\boldsymbol{1} + \psi_2^\delta (\nabla \psi_1 - \nabla \psi_2)\cdot \boldsymbol{1}  \right).
\end{align}}
To estimate the second term, we use generealized H\"{o}lder's inequality as 
\begin{align*}
	\|\psi_2^\delta (\nabla \psi_1 - \nabla \psi_2)\cdot \boldsymbol{1}\|_{L^2(0,\infty; L^2(\Omega))}^2 & = \int_0^\infty \| \psi_2^\delta(t) \nabla (\psi_1  -  \psi_2 )(t)\cdot \boldsymbol{1}\|^2 dt \\
	& \le C \int_0^\infty  \| \psi_2(t)^\delta\|_{L^3(\Omega)}^2 \|\nabla (\psi_1(t) -  \psi_2(t))\|_{L^6(\Omega)}^2 dt\\
	& \le C \|\psi_2\|_{L^{\infty}(0,\infty; L^{3\delta}(\Omega))}^{2\delta} \|\psi_1-\psi_2\|_{L^2(0,\infty; H^2(\Omega))}^2.
\end{align*}
Let us set the function $\varphi(\psi)=u^{\psi}$. Then, we have 
\begin{align}\label{aux_res1}
\psi_1^{\delta}-\psi_2^{\delta}&=\varphi(\psi_1)-\varphi(\psi_2)=\int_0^1\frac{d}{d\theta}\varphi(\theta\psi_1+(1-\theta)\psi_2)d\theta\no\\&=\int_0^1\varphi'(\theta\psi_1+(1-\theta)\psi_2)(\psi_1-\psi_2)d\theta\no\\
&=\delta\int_0^1(\theta\psi_1+(1-\theta)\psi_2)^{\delta-1}(\psi_1-\psi_2)d\theta.
\end{align}
Now we estimate the first term in the right hand side of \eqref{eqn-diff}  using generalized H\"older's inequality and Sobolev's embedding as 
\begin{align*}
 \|(\psi_1^\delta - \psi_2^\delta) \nabla \psi_1 & \cdot\boldsymbol{1}\|_{L^2(0,\infty; L^2(\Omega))}^2  = \int_0^\infty \| (\psi_1^\delta(t) - \psi_2^\delta(t)) \nabla \psi_1(t)\cdot\boldsymbol{1}\|^2 dt \\
 & = \int_0^\infty  \bigg\|(\psi_1(t) - \psi_2(t))\bigg( \int_0^1(\theta\psi_1(t) +(1-\theta)\psi_2(t))^{\delta-1} \, d\theta\bigg)\nabla \psi_1(t)\cdot\boldsymbol{1}\bigg\|^2 dt\\
 & \le C \int_0^\infty \|\psi_1(t) -\psi_2(t)\|_{L^6(\Omega)}^2 \| |\psi_1(t) |^{\delta-1}+ |\psi_2(t)|^{\delta-1}\|_{L^6(\Omega)}^2 \|\nabla \psi_1(t)\|_{L^6(\Omega)}^2 dt \\
 & \le C \|\psi_1 - \psi_2\|_{L^\infty(0,\infty; H^1_0(\Omega))}^2 \Big( \|\psi_1\|_{L^\infty(0,\infty; L^{6(\delta-1)}(\Omega))}^{2(\delta-1)} \\
 & \qquad \qquad  + \|\psi_2\|_{L^\infty(0,\infty; L^{6(\delta-1)}(\Omega))}^{2(\delta-1)} \Big) \| \psi_1\|_{L^2(0,\infty; H^2(\Omega))}^2.
\end{align*}
Combining the above two estimates, one can conclude the proof of part (a).

\medskip
\noindent (b) Note that 
\begin{align*}
    g(\psi_1) - g(\psi_2) & =  \beta \psi_1^{2\delta+1}-\beta (1+\gamma)\psi_1^{\delta+1} - \beta \psi_2^{2\delta+1}+\beta (1+\gamma)\psi_2^{\delta+1} \\
    & = \beta ( \psi_1^{2\delta+1} - \psi_2^{2\delta+1}) -  \beta (1+\gamma)(\psi_1^{\delta+1}-\psi_2^{\delta+1}). 
\end{align*}
Let us estimate the above terms one by one. Considering the fact \eqref{aux_res1} and using generalized H\"{o}lder's inequality, we get 
\begin{align*}
  &  \left\| \psi_1^{2\delta+1} - \psi_2^{2\delta+1}\right\|_{L^2(0,\infty; L^2(\Omega))}^2 \no\\& = \int_0^\infty \left\| \left(\psi_1(t) - \psi_2(t)\right) \int_0^1(\theta \psi_1(t) +(1-\theta) \psi_2(t))^{2\delta}\, d\theta\right\|^2 dt \\
    & \le C \int_0^\infty \big\| \psi_1(t) - \psi_2(t) \big\|_{L^6(\Omega)}^2  \left( \big\| \psi_1^{2\delta}(t) \big\|_{L^3(\Omega)}^2 + \big\| \psi_2^{2\delta}(t) \big\|_{L^3(\Omega)}^2 \right) dt \\
    & \le C_s \|\psi_1-\psi_2\|_{L^\infty(0,\infty;H^1_0(\Omega))}^{2} \int_0^\infty\left( \|\psi_1(t)\|_{ L^{6\delta}(\Omega)}^{4\delta} + \|\psi_2(t)\|_{ L^{6\delta}(\Omega)}^{4\delta}\right)dt.\\
\end{align*}
Due to Sobolev's embedding, for $\delta=1,$ or $d=2,$ the above estimate is enough. However for $\delta= 2$ and $d=3,$  we use interpolation inequality  to have
\begin{align*}
	\int_0^\infty  \|\psi_i(t)\|_{ L^{6\delta}(\Omega)}^{4\delta} dt & \le C \int_0^\infty \|\psi_i(t)\|_{ L^{6(\delta-1)}(\Omega)}^{2\delta} \|\psi_i(t)\|_{ L^{6(\delta+1)}(\Omega)}^{2\delta} dt \\
	& \le C \int_0^\infty \|\psi_i(t)\|_{ L^{6(\delta-1)}(\Omega)}^{2(\delta-1)} \|\psi_i(t)\|_{ L^{6(\delta+1)}(\Omega)}^{2(\delta+1)}dt  \\
	&\le C_s \|\psi_i\|_{L^\infty (0,\infty; H^1_0(\Omega))}^{2(\delta-1)} \|\psi_i\|_{L^{2(\delta+1)} (0,\infty; L^{6(\delta+1)}(\Omega))}^{2(\delta+1)}, \quad i=1,2.
\end{align*}
Similarly, one can obtain
\begin{align*}
	\left\| \psi_1^{\delta+1} - \psi_2^{\delta+1}\right\|_{L^2(0,\infty; L^2(\Omega))}^2 & = \int_0^\infty \left\| \left(\psi_1(t)- \psi_2(t)\right)\int_0^1 (\theta \psi_1(t) +(1-\theta) \psi_2(t))^{\delta}\, d\theta\right\|^2 dt \\
	& \le  \int_0^\infty \big\| \psi_1(t) - \psi_2(t) \big\|_{L^6(\Omega)}^2  \left( \big\| \psi_1^{\delta}(t)\big\|_{L^3(\Omega)}^2 + \big\| \psi_2^{\delta}(t) \big\|_{L^3(\Omega)}^2 \right) dt \\
	& \le C_s \|\psi_1-\psi_2\|_{L^\infty(0,\infty;H^1_0(\Omega))}^{2} \int_0^\infty\left( \|\psi_1(t)\|_{ L^{3\delta}(\Omega)}^{2\delta} + \|\psi_2(t)\|_{ L^{3\delta}(\Omega)}^{2\delta}\right)dt. 
\end{align*}
Then,  combining the above estimates, we arrive at the required result.
\end{proof}

\begin{proof}[Proof of Theorem \ref{th:localStrongsolGBHE}]
Recall that with $z(\cdot,t)=\int_0^t e^{-\lambda(t-s)}y(\cdot,s) ds,$ we can re-write the problem \eqref{eq:GBHE-f} as
\begin{equation} \label{eq:GBHE-f-coup}
\left\{
    \begin{aligned}
& y_t - \eta \Delta y +  \alpha y\sum_{i=1}^d \frac{\partial y}{\partial x_i} - \kappa\Delta z + \beta y(y^\delta-1)(y^\delta-\gamma) = f \ \text{ in }\  \Omega\times (0,T), \\
& z_t+\lambda z - y =0 \ \text{ in }\  \Omega\times (0,T),\\
& y=0 , \, \& \,z=0  \ \text{ on }\  \Gamma\times (0,T), \\
        & y(0)=y_0, \, \& \, z(0)=0 \ \text{ in }\ \Omega.
    \end{aligned}\right.
\end{equation}
\textbf{Step. 1.} Let us first focus on the corresponding linear equation, 
\begin{equation} \label{eq:GBHE-f-coupLin}
\left\{
    \begin{aligned}
& \bar y_t - \eta \Delta \bar y +\beta\gamma \bar y - \kappa\Delta \bar z =\bar f \ \text{ in }\  \Omega\times (0,\infty), \\
& \bar z_t+\lambda\bar z - \bar y =0 \ \text{ in }\  \Omega\times (0,\infty),\\
& \bar y=0 , \, \& \,\bar z=0  \ \text{ on }\  \Gamma\times (0,\infty), \\
        & \bar y(0)=\bar y_0, \, \& \, \bar z(0)=0 \ \text{ in }\ \Omega.
    \end{aligned}\right.
\end{equation}
The above equation can be written in the form of \eqref{eq:opform-f}, and hence using Theorem \ref{th:welPosedLinOp}(b), we obtain a unique strong solution $\begin{pmatrix}\bar{y}\\ \bar{z} \end{pmatrix} \in L^2(0,\infty;\Hf)$ of \eqref{eq:GBHE-f-coupLin}. Now, to prove the regularity of $\bar{y}, \bar{z},$ set $v(x,t)=\bar y(x,t)+\frac{\kappa}{\eta} \bar z(x,t).$ Then from \eqref{eq:GBHE-f-coupLin}, one can easily obtain 
\begin{equation} \label{eq:GBHE-f-coupLinHt}
\left\{
    \begin{aligned}
& v_t - \eta \Delta v=\bar f+ \left(\frac{\kappa}{\eta} - \beta\gamma\right)\bar y-\frac{\lambda\kappa}{\eta} \bar z=:F   \ \text{ in }\  \Omega\times (0,\infty), \\
& v=0 ,  \ \text{ on }\  \Gamma\times (0,\infty), \\
        & \bar v(0)= \bar y_0 \ \text{ in }\ \Omega.
    \end{aligned}\right.
\end{equation}
Since, we already have $\bar y, \bar z \in L^2(0,\infty;L^2(\Omega))$, and hence $F\in L^2(0,\infty;L^2(\Omega)).$ Therefore, for $\bar y_0 \in H^1_0(\Omega),$ Proposition \ref{pps:heatStrongL2L6} yields the existence of a strong unique solution $v$ of \eqref{eq:GBHE-f-coupLinHt} such that $ v\in L^2(0,\infty; H^2(\Omega)) \cap L^\infty(0,\infty; H^1_0(\Omega)) \cap L^{2(\delta+1)}(0,\infty; L^{6(\delta+1)}(\Omega))$ with $v_t \in L^2(0,\infty; L^2(\Omega)),$ and the following estimate holds:
\begin{equation} \label{eq:est-linCoup-C1}
\begin{aligned}
\|v\|_{L^2(0,\infty; H^2(\Omega))} + \|v\|_{L^\infty(0,\infty; H^1_0(\Omega))} + &\|v\|_{L^{2(\delta+1)}(0,\infty; L^{6(\delta+1)}(\Omega))} + \|v\|_{H^1(0,\infty;L^2(\Omega))} \\
& \le C_1(\|\bar y_0\|_{H^1_0(\Omega)} +\|\bar f\|_{L^2(0,\infty;L^2(\Omega))}),
\end{aligned}
\end{equation}
for some $C_1>0.$ Furthermore, it is easy to observe that
\begin{equation}
	\left\{
\begin{aligned}
& \bar z_t(x,t)  =-\left(\frac{\kappa}{\eta} +\lambda\right)\bar z(x,t) +v(x,t) \ \text{ in }\ \Omega\times (0,\infty), \\
& \bar z(x,t)=0 \ \text{ on }\  \Gamma\times (0,\infty), \\
& \bar z(x,0)=0 \ \text{ in }\  \Omega.
\end{aligned}
\right.
\end{equation}
Then, one can write the solution of the above equation as
\begin{align*}
\bar z(x,t)=\int_0^t e^{- (\frac{\kappa}{\eta} +\lambda) (t-s)}v(x,s) ds\  \text{ for all }\ (x,t) \in \Omega\times (0,\infty).
\end{align*}
Thus, we can easily get $\bar z\in H^1(0,\infty; H^2(\Omega)\cap H^1_0(\Omega)) \cap L^{2(\delta+1)}(0,\infty; L^{6(\delta+1)}(\Omega))$ and hence we obtain a unique strong solution  $\bar y \in L^2(0,\infty; H^2(\Omega)) \cap L^\infty(0,\infty; H^1_0(\Omega)) \cap L^{2(\delta+1)}(0,\infty; L^{6(\delta+1)}(\Omega))$ of \eqref{eq:GBHE-f-coupLin}.\\

\noindent \textbf{Step. 2.} In this step, using Propositions \ref{pps:GBHE-selfmap-I} - \ref{pps:GBHE-contrac-I}, we apply the Banach fixed-point theorem to conclude the result. For given $\psi \in D,$ let $y^\psi, z^\psi$ satisfy 
\begin{equation} \label{eq:GBHE-gh-coup}
\left\{
    \begin{aligned}
& y^\psi_t - \eta \Delta y^\psi +\beta\gamma y^\psi  - \kappa\Delta z^\psi = f +g(\psi)+h(\psi), \ \text{ in }\  \Omega\times (0,\infty), \\
& z^\psi_t+\lambda z^\psi - y^\psi =0 \ \text{ in }\  \Omega\times (0,\infty),\\
& y^\psi=0 , \, \& \,z^\psi=0  \ \text{ on }\  \Gamma\times (0,\infty), \\
        & y^\psi(0)=y_0, \, \& \, z^\psi(0)=0 \ \text{ in }\ \Omega,
    \end{aligned}\right.
\end{equation}
where $h$ and $g$ be as given in \eqref{eq:NL-h-GBHE} and \eqref{eq:NL-g-GBHE}, respectively. From Proposition \ref{pps:GBHE-selfmap-I}, it is clear that $h(\psi), g(\psi) \in L^2(0,\infty; L^2(\Omega)),$ consequently, we can apply Step 1 to get a unique strong solution $y^\psi \in D$ and $z^\psi \in H^1(0,\infty; H^2(\Omega)\cap H^1_0(\Omega))$ satisfying
\begin{align*}
    \|y^\psi\|_D+\|z^\psi\|_{H^1(0,\infty; H^2(\Omega))} \le & C_1 \Big( \|y_0\|_{H^1_0(\Omega)}  +\|f\|_{L^2(0,\infty; L^2(\Omega))} \\
    & +\|h(\psi)\|_{L^2(0,\infty; L^2(\Omega))}+\|g(\psi)\|_{L^2(0,\infty; L^2(\Omega))}\Big),
\end{align*}
where $C_1>0$ is as in \eqref{eq:est-linCoup-C1}. Now, for any $\rho>0,$ let us define the balls
\begin{equation} \label{eq:D_rho - B_rho}
\begin{aligned} 
    D_\rho=\{\psi \in D\, |\, \|\psi\|_D \le \rho \} \text{ and } B_\rho =\{ \psi \in H^1(0,\infty; H^2(\Omega)\cap H^1_0(\Omega)) \, |\, \|\psi\|_{H^1(0,\infty; H^2(\Omega))}\le \rho \}.
\end{aligned}
\end{equation}
For a given $\rho>0,$ let $\mathcal{S}_1: D_\rho \to D_\rho \times B_\rho $ be the solution map such that $$\mathcal{S}_1(\mathcal{\psi})=(y^\psi, z^\psi),$$ where for a given $\psi\in D,$ $(y^\psi, z^\psi)$ is the solution pair of the system \eqref{eq:GBHE-gh-coup}. And, we define the map $\mathcal{S}_2: D_\rho \times B_\rho \to D_\rho $ by $$\mathcal{S}_2 (y^\psi, z^\psi)= y^\psi. $$ Let $\mathcal{S}:=\mathcal{S}_2\circ \mathcal{S}_1: D_\rho \rightarrow D_\rho.$ To show the map $\mathcal{S}$ is well-defined, we assume that $\|y_0\|_{H^1_0(\Omega)}  \le \frac{\rho}{4C_1}$ and $\|f\|_{L^2(0,\infty; L^2(\Omega))} \le \frac{\rho}{4C_1}.$
Now, we obtain 
\begin{align*}
    \|y^\psi\|_D+\|z^\psi\|_{H^1(0,\infty; H^2(\Omega))} & \le C_1 \Big( \|y_0\|_{H^1_0(\Omega)}   +\|f\|_{L^2(0,\infty; L^2(\Omega))} \\
    & \qquad \qquad  +\|h(\psi)\|_{L^2(0,\infty; L^2(\Omega))}+\|g(\psi)\|_{L^2(0,\infty; L^2(\Omega))}\Big). \\
    & \le C_1 \left( \frac{\rho}{4C_1} +  \frac{\rho}{4C_1} +  2M_1\rho^{\delta+1}+M_1\rho^{2\delta+1}    \right) \le \rho,
 \end{align*} 
for $\rho \le \min\left\lbrace \frac{1}{(8C_1M_1)^{\frac{1}{\delta}}}, \frac{1}{(4C_1M_1)^{\frac{1}{2\delta}}} \right\rbrace.$ Thus $\mathcal{S}=\mathcal{S}_2\circ\mathcal{S}_1: D_\rho \to D_\rho$ is well defined for all $0<\rho \le \min\left\lbrace \frac{1}{(8C_1M_1)^{\frac{1}{\delta}}}, \frac{1}{(4C_1M_1)^{\frac{1}{2\delta}}} \right\rbrace.$ 
Next, we aim to show that $\mathcal{S}: D_\rho \to D_\rho$ is contraction for all $0<\rho\le \rho_1$ for some $\rho_1>0.$

Let  for $\psi_1, \psi_2 \in D_\rho,$ $(y^{\psi_1},z^{\psi_1})$ and $(y^{\psi_2},z^{\psi_2})$ be the corresponding solutions of \eqref{eq:GBHE-gh-coup}. Then $\yb=y^{\psi_1} - y^{\psi_2}$ and $\zb =z^{\psi_1} - z^{\psi_2}$ satisfy

\begin{equation} \label{eq:GBHE-yb-zb-coup}
\left\{
    \begin{aligned}
& \yb_t - \eta \Delta \yb +\beta\gamma \yb - \kappa\Delta \zb =  g(\psi_1)-g(\psi_2)+h(\psi_1)- h(\psi_2), \ \text{ in }\  \Omega\times (0,\infty), \\
& \zb_t+\lambda \zb - \yb =0 \ \text{ in }\  \Omega\times (0,\infty),\\
& \yb=0 , \, \& \, \zb=0  \ \text{ on }\  \Gamma\times (0,\infty), \\
& \yb(0)=0, \, \& \, \zb(0)=0 \ \text{ in }\ \Omega.
    \end{aligned}\right.
\end{equation}
Now, again using Step 1 and Propositions \ref{pps:GBHE-selfmap-I} - \ref{pps:GBHE-contrac-I}, we obtain 
\begin{align*}
    \|\yb\|_D+\|\zb\|_{H^1(0,\infty; H^2(\Omega))} & \le C_1 \Big(  \|h(\psi_1)-h(\psi_2)\|_{L^2(0,\infty; L^2(\Omega))}+\|g(\psi_1)-g(\psi_2)\|_{L^2(0,\infty; L^2(\Omega))}\Big) \\
    & \le C_1 M_2 \left( 2\rho^{2\delta} +5\rho^{\delta}    \right)  \|\psi_1 - \psi_2\|_D .
\end{align*}
Now, choosing 
\begin{align}\label{eq:rho0}
\rho_1=\min\left\lbrace \frac{1}{(8C_1M_1)^{\frac{1}{\delta}}}, \frac{1}{(4C_1M_1)^{\frac{1}{2\delta}}}, \frac{1}{(6C_1M_2)^{\frac{1}{2\delta}}}, \frac{1}{(15C_1M_2)^{\frac{1}{\delta}}} \right\rbrace,
\end{align}
 we obtain 
\begin{align*}
    \|y^{\psi_1} - y^{\psi_2}\|_D \le \frac{2}{3}\|\psi_1 - \psi_2\|_D,
\end{align*}
and hence $\mathcal{S}$ is a contraction map. Thus using the Banach fixed point theorem, we obtain a unique strong solution 
\begin{align*}
& y \in L^2(0,\infty; H^2(\Omega))\cap L^\infty(0,\infty; H^1_0(\Omega)) \cap L^{2(\delta+1)}(0,\infty; L^{6(\delta+1)}(\Omega))\cap H^1(0,\infty; L^2(\Omega)), \\
& z\in 	H^1(0,\infty; H^2(\Omega)\cap H^1_0(\Omega)),
\end{align*}
for all $y_0\in H^1_0(\Omega)$ and $f \in L^2(0,\infty; L^2(\Omega))$ with 
\begin{align*}
    \|y_0\|_{H^1_0(\Omega)}  \le \frac{\rho}{4C_1} \text{ and } \|f\|_{L^2(0,\infty; L^2(\Omega))} \le \frac{\rho}{4C_1},
\end{align*}
for all $0<\rho \le \rho_1,$ where $\rho_1$ as in \eqref{eq:rho0}.

\medskip 
\noindent \textbf{Step 3.} (Global solution). In this step, we first aim to show that 
\begin{align*}
	\int_0^T \||y(s)|^\delta \nabla y(s)\|^2 ds \le C \left( \|y_0\|_{H^1_0(\Omega)}^2+\|f\|_{L^2(0,T; L^2(\Omega))}^2 \right),
\end{align*}
for some $C>0.$ To do so, we take an inner product in the first equation of \eqref{eq:GBHE-f-coup} with $|y|^{2\delta}y$ to find 
\begin{equation} \label{eq:GbWkest-2}
	\begin{aligned}
		(y_t, |y|^{2\delta}y) - \eta (\Delta y, |y|^{2\delta}y) & +  (\alpha y^\delta \nabla y \cdot \boldsymbol{1}, |y|^{2\delta}y) -\kappa (\Delta z, |y|^{2\delta}y) +(\beta y^{2\delta +1} , |y|^{2\delta}y ) + (\beta \gamma y, |y|^{2\delta}y ) \\
		& =   (\beta(1+\gamma) y^{\delta+1}, |y|^{2\delta}y ) + (f , |y|^{2\delta}y).
	\end{aligned}
\end{equation}
Note that for sufficiently small initial data,  $|y|^\delta y\in L^2(0,\infty; H^1_0(\Omega))$ and hence using absolute continuity lemma \cite[Theorem 3, Chapter 5, Section 5.9]{Eva} (Lions-Magenes lemma), we obtain for a.e. $t\in[0,T]$
\begin{align*}
(y_t(t), |y(t)|^{2\delta}y(t))=(|y(t)|^{\delta}y_t(t), |y(t)|^{\delta}y(t)) = \frac{1}{2(\delta+1)} \frac{d}{dt}\|y(t)\|_{L^{2(\delta+1)}(\Omega)}^{2(\delta+1)}.
\end{align*}
Also, one can notice the following:
\begin{align*}
	& - \eta (\Delta y, |y|^{2\delta}y) = \eta (\nabla y, (2\delta+1)|y|^{2\delta}\nabla y) = \eta (2\delta+1) \||y|^\delta \nabla y|^2,\\
	& (\alpha y^\delta \nabla y \cdot \boldsymbol{1}, |y|^{\delta}y)=0,\\
	& -\kappa (\Delta z, |y|^{2\delta} y) = \kappa (2\delta+1)  \left( \int_0^t e^{-\lambda(t-s)}\nabla y(x,s) \,ds, |y(x,t)|^{2\delta}\nabla y(x,t)\right) \\
	&  \hspace{2.6cm} = \kappa (2\delta+1) \left( |y(x,t)|^{2\delta} \int_0^t e^{-\lambda(t-s)}\nabla y(x,s) \,ds,  \nabla y(x,t)\right) \ge 0,\\
	& (\beta y^{2\delta +1} , |y|^{2\delta} y ) = \beta \||y|^{2\delta} y\|^2 = \beta\| y\|_{L^{4\delta+2}}^{4\delta+2}, \\
	& (\beta \gamma y, |y|^{2\delta}y ) =\beta \gamma (|y|^{\delta}y, |y|^{\delta}y)=\beta \gamma \||y|^{\delta}y\|^2= \beta \gamma\|y\|_{L^{2(\delta+1)}(\Omega)}^{2(\delta+1)}\\
	& (\beta(1+\gamma) y^{\delta+1}, |y|^{2\delta}y ) \le \frac{\beta}{4}\|y\|_{L^{4\delta+2}(\Omega)}^{4\delta+2}  + \beta (1+\gamma)^2 \|y\|_{L^{2(\delta+1)}(\Omega)}^{2(\delta+1)}, \\
	&| (f , |y|^{2\delta}y)|  \le \frac{\beta}{4}\|y\|_{L^{4\delta+2}(\Omega)}^{4\delta+2} + \frac{1}{\beta}\|f\|^2.
\end{align*}
where in the third estimate we have used  \cite[Lemma A.2]{MTMSSS}.  Utilization of the above estimates in \eqref{eq:GbWkest-2} leads to 
\begin{align*}
&	\frac{1}{2(\delta+1)} \frac{d}{dt}\|y(t)\|_{L^{2(\delta+1)}(\Omega)}^{2(\delta+1)} + \eta (2\delta+1) \||y|^\delta \nabla y|^2 + \frac{\beta}{2}\|y\|_{L^{4\delta+2}(\Omega)}^{4\delta+2} \\&\le \beta (1+\gamma)^2 \|y\|_{L^{2(\delta+1)}(\Omega)}^{2(\delta+1)} + \frac{1}{\beta}\|f\|^2,
\end{align*}
which further on integration can be written as 
\begin{align*}
	&\frac{1}{2(\delta+1)} \|y(t)\|_{L^{2(\delta+1)}(\Omega)}^{2(\delta+1)} + \eta (2\delta+1) \int_0^t\||y(s)|^\delta \nabla y(s)|^2\, ds + \frac{\beta}{2} \int_0^t\|y(s)\|_{L^{4\delta+2}(\Omega)}^{4\delta+2}\, ds \\
	& \le \frac{1}{2(\delta+1)} \|y_0\|_{L^{2(\delta+1)}(\Omega)}^{2(\delta+1)} +\beta (1+\gamma)^2\int_0^t \|y(s)\|_{L^{2(\delta+1)}(\Omega)}^{2(\delta+1)}\,ds + \frac{1}{\beta}\int_0^t\|f(s)\|^2\, ds.
\end{align*}
An application of Gronwalls' inequality further gives 
\begin{equation} \label{eq:est4ydgrad6}
\begin{aligned}
	\sup_{0\le t\le T} \|y(t)\|_{L^{2(\delta+1)}(\Omega)}^{2(\delta+1)} & + 2\eta(\delta+1) (2\delta+1) \int_0^T\||y(s)|^\delta \nabla y(s)|^2\, ds + \beta(\delta+1) \int_0^T\|y(s)\|_{L^{4\delta+2}(\Omega)}^{4\delta+2}\, ds \\
	& \le C \left( \|y_0\|_{H^1_0(\Omega)}^{2(\delta+1)}+ \frac{2(\delta+1)}{\beta}\|f\|_{L^2(0,T; L^2(\Omega))}^2 \right)e^{\beta(1+\gamma)^2T}.
\end{aligned}
\end{equation}

Next, taking the inner-product with $-\Delta y$ in the first equation of \eqref{eq:GBHE-f-coup}, we obtain for a.e. $t\in[0,T]$
\begin{equation} \label{eq:GbWkest}
\begin{aligned}
&\frac{1}{2}\frac{d}{dt}\|\nabla y(t)\|^2 + \eta \|\Delta y(t)\|^2 +\kappa (\Delta z(t), \Delta y(t)) -(\beta y^{2\delta +1}(t) , \Delta y(t) ) \\
& =  (\alpha y^\delta(t) \nabla y(t) \cdot \boldsymbol{1}, \Delta y(t)) - (\beta(1+\gamma) y^{\delta+1}(t) +\beta \gamma y(t), \Delta y(t) ) - (f(t) , \Delta y(t)).
\end{aligned}
\end{equation}
One can easily observe the following:
\begin{align*}
& \kappa (\Delta z, \Delta y) = \kappa \int_0^t e^{-\lambda(t-s)}\Delta y(x,s) \Delta y(x,t) \, ds \ge 0, \\
& -(\beta y^{2\delta +1} , \Delta y ) = \beta (2\delta +1) (y^{2\delta} \nabla y, \nabla y) = \beta (2\delta +1) \|y^\delta \nabla y\|^2, \\
 & |(\alpha y^\delta \nabla y \cdot \boldsymbol{1}, \Delta y)| \le \alpha \| y^\delta \nabla y\| \|\Delta y\| \le \frac{\eta}{4}\|\Delta y\|^2 + \frac{\alpha^2}{\eta}  \| y^\delta \nabla y\|^2,\\
 & | (\beta(1+\gamma) y^{\delta+1} , \Delta y )  | \le  \beta(1+\gamma) (\delta +1) \|y^\delta \nabla y\| \|  \nabla y\| \le \frac{\beta (2\delta+1)}{2} \|y^\delta \nabla y\|^2 + C\|\nabla y\|^2 \\
 & (\beta \gamma y, \Delta y ) = -\beta \gamma \|\nabla y\|^2, \text{ and } | (f , \Delta y)|  \le \frac{\eta}{4}\|\Delta y\|^2 + \frac{1}{\eta}\|f\|^2,
\end{align*}
where in the first estimate we have used \cite[Lemma A.2]{MTMSSS}. Utilization of the above estimates in \eqref{eq:GbWkest} leads to 
\begin{align*}
  &  \frac{d}{dt}\|\nabla y(t)\|^2 +\eta \|\Delta y(t)\|^2 + \beta (2\delta+1)  \|y^\delta(t) \nabla y(t)\|^2\\& \le C\|\nabla y(t)\|^2 +  2\frac{\alpha^2}{\eta}  \|y^\delta(t) \nabla y(t)\|^2 +\frac{1}{\eta}\|f(t)\|^2.
\end{align*}
Performing integration over time in $(0,T)$ and applying Gronwall's inequality, one obtains
\begin{align*}
  \sup_{0\le t\le T}\|\nabla y(t)\|^2 & +\eta \int_0^T \|\Delta y(t)\|^2 dt +   \beta (2\delta+1)  \int_0^T  \|y^\delta(t) \nabla y(t)\|^2 dt \\
  & \le C\left ( \|\nabla y_0\|^2 + 2\frac{\alpha^2}{\eta}  \int_0^T  \|y^\delta(t) \nabla y(t)\|^2 dt +\frac{1}{\eta} \|f\|_{L^2(0,T; L^2(\Omega))}^2 \right).
\end{align*}
We conclude the proof by using \eqref{eq:est4ydgrad6} by noting that 
\begin{align*}
    \|y\|_{L^{2(\delta+1)}(0,T; L^{6(\delta+1)}(\Omega))}^{2(\delta+1)}  = \int_0^T \|(y(t))^{\delta+1}\|_{L^6(\Omega)}^2 dt  \le C(\delta+1)^2 \int_0^T \| (y(t))^{\delta} \nabla y(t)\|^2 dt,
\end{align*}
and the estimate \eqref{en-est} follows. 
\end{proof}

\begin{Remark}
	From Step 2, we observe that in the case of an infinite time horizon, a local solution exists in the sense that, under the smallness assumption on the initial data and forcing term, we obtain a strong solution over an infinite time horizon. However, in the case of a finite time horizon, we establish the existence of a global solution.
\end{Remark}

\section{Stabilizability around the zero state} \label{sec:stabaroundzero}
This section focuses on the stabilizability of 
\begin{equation} \label{eq:GBHE-C0}
	\left\{
	\begin{aligned}
		& y_t - \eta \Delta y +  \alpha y^\delta \sum_{i=1}^d \frac{\partial y}{\partial x_i} - \kappa \int_0^t e^{-\lambda (t-s)}\Delta y(s)ds  = \beta y(1-y^\delta)(y^\delta-\gamma)+ u\chi_\mathcal{O} \ \text{ in }\  \Omega\times (0,\infty), \\
		& y=0  \ \text{ on }\  \Gamma\times (0,\infty), \\
		& y(x,0)=y_0 \ \text{ in }\ \Omega,
	\end{aligned}\right.
\end{equation}
around the zero steady state by using a localized interior control $u$. Here, $\mathcal{O}\subset \Omega$ is a non-empty open set, and other parameters are the same as discussed in the beginning of this article. Similar to the previous section, with $z(\cdot,t)=\int_0^t e^{-\lambda(t-s)}y(\cdot,s) ds,$ one can re-write the system  \eqref{eq:GBHE} as
\begin{equation} \label{eq:GBHE-u-coup-Z}
	\left\{
	\begin{aligned}
		& y_t - \eta \Delta y +  \alpha y^\delta\sum_{i=1}^d \frac{\partial y}{\partial x_i} - \kappa\Delta z + \beta y(y^\delta-1)(y^\delta-\gamma) = u\chi_\mathcal{O} \ \text{ in }\  \Omega\times (0,\infty), \\
		& z_t+\lambda z - y =0 \ \text{ in }\  \Omega\times (0,\infty),\\
		& y=0 , \, \& \,z=0  \ \text{ on }\  \Gamma\times (0,\infty), \\
		& y(0)=y_0, \, \& \, z(0)=0 \ \text{ in }\ \Omega.
	\end{aligned}\right.
\end{equation}
The corresponding linear system  is
\begin{equation} \label{eq:GBHE-Lin-coup-Z}
	\left\{
	\begin{aligned}
		& y_t - \eta \Delta y +\beta\gamma y - \kappa\Delta z = u\chi_\mathcal{O} \ \text{ in }\  \Omega\times (0,\infty), \\
		& z_t+\lambda z - y =0 \ \text{ in }\  \Omega\times (0,\infty),\\
		& y=0 , \, \& \,z=0  \ \text{ on }\  \Gamma\times (0,\infty), \\
		& y(0)=y_0, \, \& \, z(0)=0 \ \text{ in }\ \Omega.
	\end{aligned}\right.
\end{equation}
In $\Hf=L^2(\Omega)\times L^2(\Omega),$ and with $\Wf(\cdot,t) =\begin{pmatrix} y(\cdot,t) \\ z(\cdot,t) \end{pmatrix},$ the system \eqref{eq:GBHE-Lin-coup-Z} can be re-written as 
\begin{equation} \label{eq:linOpCtrl}
	\Wf'(t)=\Af \Wf(t) + \Bf u(t), \text{ for all }t>0, \quad \Wf(0)= \begin{pmatrix} w_0 \\ 0 \end{pmatrix}=:\boldsymbol{w}_0,
\end{equation}
where the unbounded operator $(\Af, D(\Af))$ is as defined in \eqref{eqdef:A} and the control operator $\Bf : L^2(\Omega) \rightarrow \Hf $ is defined as 
\begin{equation} \label{eqdef:B}
\Bf u= \begin{pmatrix} u \chi_{\mathcal{O}} \\ 0 \end{pmatrix} \text{ for all } u \in L^2(\Omega).
\end{equation}
The corresponding linear system, represented by the operator form \eqref{eq:linOp}, is identical to that studied in \cite{WKR} with $\kappa=1$ and $\beta=0$ or $\gamma=0$. Note that the eigenfunctions, we choose in \eqref{eq:eigfun-A} and \eqref{eq:eigfun-A*} are the same if $\kappa=1.$  Consequently, the analysis provided in Sections 3 and 4 of \cite{WKR} remains applicable to the current scenario. As a result, the findings for the linear system can be directly extended here. To avoid redundancy, we omit the detailed analysis and instead state the main result for the linear system in our context, before proceeding to establish the results for the non-linear system.

\subsection{Stabilizability of the linear system} 
We first aim to study the stabilizability of the system \eqref{eq:GBHE} with decay $\nu \in (0,\nu_0),$ where $\nu_0$ is as in \eqref{eq:nu0} and hence it is convenient to study the stabilizability of the shifted system:
\begin{equation} \label{eq:siftLinSysOP}
\wt{\Wf}'(t)=\Af_\nu \wt{\Wf}(t) + \Bf \wt{u}(t) \quad \text{ for }t>0, \qquad \wt{\Wf}(0)=\Wf_0,
\end{equation}
where $\wt{\Wf}(t)=e^{\nu t}\Wf(t),$ $\wt{u}(t)=e^{\nu t}u(t),$ and
\begin{align} \label{eqdef:A_nu}
    \Af_\nu=\Af+\nu \boldsymbol{I}, \text{ with } D(\Af_\nu)=D(\Af),
\end{align}
with $\boldsymbol{I}:\Hf \to \Hf$ being the identity map. Note that the spectrum of $\Af_\nu,$ denoted by $\sigma(\Af_\nu),$ is $\sigma(\Af_\nu) = \{ \mu_k^++\nu, \mu_k^-+\nu\, |\, k \in \mathbb{N} \} \cup \{ \nu -\nu_0\}.$
From Theorem \ref{th:specAna}, it follows that there exists $N_\nu\in \mathbb{N}$ such that
\begin{align} \label{eqref:SpecCond22}
    \Re(\mu_k^\pm+\nu) <0 \text{ for all }k > N_\nu, \text{ and } \Re(\mu_k^\pm+\nu) \ge 0 \text{ for all }1\le k\le N_\nu .
\end{align}
Consequently, we define
\begin{align} \label{eq:specA_lam+-}
    \sigma_+(\Af_\nu)=\{\mu_k^\pm +\nu \, |\, \Re(\mu_k^\pm +\nu)\ge 0\} \text{ and } \sigma_-(\Af_\nu)=\{\mu_k^\pm +\nu \, |\, \Re(\mu_k^\pm +\nu) < 0\},
\end{align}
and note that $\sigma_+(\Af_\nu)$ is a finite set. Now, one can follow \cite[Proposition 4.1]{WKR} using \cite[Theorem 15.2.1]{Tucs} to establish the stabilizability of the linear system \eqref{eq:linOpCtrl}. Often, the feedback operator $K$ is obtained by studying an optimization problem and by using an algebraic Riccati equation. 
To obtain the feedback operator, consider the following optimal control problem:
\begin{align}\label{eqoptinf_vort}
\min_{\wt{u}\in E_{\Wf_0}} J( \wt{\Wf},\wt{u}) \text{ subject to \eqref{eq:siftLinSysOP}},
\end{align}
where
\begin{equation}
J( \wt{\Wf},\wt{u}):=\int_0^\infty\big( \|\wt{\Wf}(t)\|_{\Hf}^2 + \|\wt{u}(t)\|^2\big) \, dt,
\end{equation}
and $$E_{\Wf_0}:=\{ \wt{u}\in L^2(0,\infty; L^2(\Omega))\mid \wt{\Wf} \text{ solution of \eqref{eq:siftLinSysOP} with control }\wt{u} \text{ such that } J( \wt{\Wf},\wt{u})<\infty\}.$$
The next theorem yields the minimizer of the problem \eqref{eqoptinf_vort} as well as the stabilizing control in the feedback form.  

 \begin{Theorem}[{\normalfont \cite[Theorem 2.2]{WKR}}]\label{th:stb cnt}
Let $\nu \in (0, \nu_0)$ be  any real number. Let $\Af_\nu$ (resp. $\Bf$) be as defined in \eqref{eqdef:A_nu} (resp. \eqref{eqdef:B}). Then the following results hold: 
	\begin{enumerate}
 	\item[(a)] There exists a unique operator $\normalfont\Pf\in \mathcal{L}(\Hf)$ that satisfies the \emph{non-degenerate Riccati equation}
				\begin{equation}\label{eqn:ARE}
					\begin{array}{l}
						\normalfont\Af_\nu^*\Pf+\Pf\Af_\nu-\Pf\Bf\Bf^*\Pf+\boldsymbol{I}=0,\quad \Pf=\Pf^* \geq 0 \ \text{ on }\ \Hf.
					\end{array}
				\end{equation}
	\item[(b)] For any $\Wf_0\in \Hf$, there exists a unique optimal pair $\normalfont(\Wf^\sharp,u^\sharp)$ for \eqref{eqoptinf_vort}, where for all $t>0$, $\Wf^\sharp(t)$ satisfies the \emph{closed loop system}
		\begin{equation}\label{eqcl-loop}
			\Wf{^\sharp}'(t)=(\Af_\nu-\Bf\Bf^*\Pf)\Wf^\sharp(t),\;\; \Wf^\sharp(0)=\Wf_0,
		\end{equation}
		$u^\sharp(t)$ can be expressed in the \emph{feedback form} as
		\begin{equation}\label{eqoptcntrl}
			\normalfont u^\sharp(t)=-\Bf^*\Pf\Wf^\sharp(t), 
		\end{equation}
		and
		$$\displaystyle\min_{\wt{u}\in E_{{\Wf}_0}}\normalfont J( \wt{\Wf},\wt{u})=J(\Wf^\sharp,u^\sharp)=( \Pf\Wf_0,\Wf_0).$$
	\item[(c)] The feedback control given in \eqref{eqoptcntrl} stabilizes the system \eqref{eq:siftLinSysOP}. In particular, the operator $\normalfont\Af_{\nu,\Pf}:=\Af_\nu-\Bf\Bf^*\Pf,$ with $\normalfont D(\Af_{\nu,\Pf})= D(\Af)$, generates an exponentially stable analytic semigroup $\n \{e^{t\Af_{\nu,\Pf}}\}_{t\ge 0}$, that is, there exist $\varsigma>0$ and $M>0$ such that 
	$$\n \|e^{t \Af_{\nu,\Pf}}\|_{\mathcal(\Hf)}\leq Me^{-\varsigma t}\  \text{ for all } \  t>0.$$
 	\end{enumerate}
	\end{Theorem}
\begin{Remark} \label{rem:feedbackop}
In the above theorem, $K=-\Bf^*\Pf$ is referred to as the \emph{feedback operator}. It encompasses both the finite-dimensional and infinite-dimensional cases, as considered in \cite{WKR}. For our subsequent analysis, the distinction between these two cases is not significant, as the analysis remains valid in both scenarios. Therefore, we simply denote $K$ as the feedback operator without specifying its dimensionality.
\end{Remark}
\subsection{Stabilizability of non-linear system}
The non-linear coupled equation \eqref{eq:GBHE-u-coup-Z} can be written as 
\begin{align} \label{eq:GBHE-NL-Op}
    \Wf'(t)= \Af \Wf(t)+ F(\Wf(t))+ \Bf u(t) \ \text{ for all }\ t>0, \quad \Wf(0)=\Wf_0,
\end{align}
where the unbounded operator $\Af$ and the control operator $\Bf$ are the same as \eqref{eqdef:A} and \eqref{eqdef:B}, respectively, and
\begin{align}
    F(\Wf(t)) = \begin{pmatrix} h(y(t)) + g(y(t))  \\ 0 \end{pmatrix},
\end{align}
where $h$ and $g$ are as in \eqref{eq:NL-h-GBHE} and \eqref{eq:NL-g-GBHE}, respectively.  Let us first present the shifted non-linear closed loop system as 
\begin{align} \label{eq:GBHE-NL-ClsLp}
    \wt{\Wf}'(t)= (\Af_\nu + \Bf K) \wt{\Wf}(t)+ \wt{F}(\wt{\Wf}(t)) \ \text{ for all }\ t>0, \quad \wt{\Wf}(0)=\Wf_0,
\end{align}
where 
\begin{align*}
    \wt{y}(t) =e^{\nu t}y(t), \, \wt{z}(t)=e^{\nu t}z(t), \, \wt{\Wf}(t)=e^{\nu t}\Wf(t), \, \wt{F}(\wt{\Wf}(t))= \begin{pmatrix} \wt{h}(\wt{y}(t)) + \wt{g}(\wt{y}(t))  \\ 0 \end{pmatrix},
\end{align*}
with
\begin{align} \label{eqdef:wt-h & wt-g}
\wt{h}(\wt{y}(t)):= \alpha e^{-\nu \delta t} \wt{y}^\delta \nabla \wt{y}\cdot \boldsymbol{1}  \quad\text{and} \quad \wt{g}(\wt{y}(t)):=\beta e^{-2\nu \delta t} \wt{y}^{2\delta+1} -\beta(1+\gamma) e^{-\nu \delta t} \wt{y}^{\delta+1}.
\end{align}
Furthermore, the system \eqref{eq:GBHE-NL-ClsLp} can be rewritten as
\begin{equation} \label{eq:GBHE-Lin-coup-shiftCls}
\left\{
    \begin{aligned}
& \wt{y}_t - \eta \Delta \wt{y} +(\beta\gamma -\nu) \wt{y} - \kappa\Delta \wt{z} = \chi_{\mathcal{O}} K(\wt{y}(t),\wt{z}(t)) + \wt{h}(\wt{y}(t)) + \wt{g}(\wt{y}(t))\ \text{ in }\  \Omega\times (0,\infty), \\
& \wt{z}_t+\lambda \wt{z} - \wt{y} -\nu \wt{z} =0 \ \text{ in }\  \Omega\times (0,\infty),\\
& \wt{y}=0 , \, \& \, \wt{z}=0  \ \text{ on }\  \Gamma\times (0,\infty), \\
        & \wt{y}(0)=y_0, \, \& \, \wt{z}(0)=0 \ \text{ in }\ \Omega.
    \end{aligned}\right.
\end{equation}
In order to prove the stability of the system  \eqref{eq:GBHE-Lin-coup-shiftCls}, we use the Banach fixed point theorem. Here, the approach is the same as in Steps 2 and 3 of the proof of Theorem \ref{th:localStrongsolGBHE}. We just provide the steps with an outline of proof instead of a detailed one. 
\vskip 0.1cm 
\noindent \textbf{Step 1.} We briefly establish the following result:
\begin{Theorem} \label{th:reg-nh-ff-Psi}
	Let $K$ be as obtained  in Theorem \ref{th:stb cnt}. Then for any $\bar f \in L^2(0,\infty; L^2(\Omega))$ and $y_0 \in H^1_0(\Omega),$ the following system
	\begin{equation}  \label{eq:GBHE_FbClsLinShift}
		\left\{
		\begin{aligned}
			& \bar y_t - \eta \Delta \bar y +(\beta\gamma -\nu)\bar y - \kappa\Delta \bar z =\chi_{\mathcal{O}}K(\bar{y},\bar{z})+ \bar f \ \text{ in }\  \Omega\times (0,\infty), \\
			& \bar z_t+(\lambda-\nu)\bar z - \bar y =0 \ \text{ in }\  \Omega\times (0,\infty),\\
			& \bar y=0 , \, \& \,\bar z=0  \ \text{ on }\  \Gamma\times (0,\infty), \\
			& \bar y(0)= y_0, \, \& \, \bar z(0)=0 \ \text{ in }\ \Omega,
		\end{aligned}
		\right. 
	\end{equation}
	has a \emph{unique strong solution} $(\bar y,\bar z)$ such that 
	\begin{align*}
		& \bar y \in \left( L^2(0,\infty; H^2(\Omega)) \cap L^\infty(0,\infty; H^1_0(\Omega)) \cap H^1(0,\infty; L^2(\Omega)) \cap L^{2(\delta+1)}(0,\infty; L^{6(\delta+1)}(\Omega)) \right),\\
		& \bar z \in H^1(0,\infty; H^2(\Omega)\cap H^1_0(\Omega)),
	\end{align*}
	satisfying
	\begin{align*}
		\|\bar y\|_D+\|\bar z\|_{H^1(0,\infty; H^2(\Omega))} \le M_3 \left( \|\bar f\|_{L^2(0,\infty;L^2(\Omega))} + \|y_0\|_{H^1_0(\Omega)} \right),
	\end{align*}
	for some $M_3>0,$ where $D$ is the same as defined in \eqref{eq:D}. 
\end{Theorem}

\begin{proof}
	The system \eqref{eq:GBHE_FbClsLinShift} can be written as 
	\begin{align*}
		W'(t)=(\Af_\nu +\Bf K) W(t) + \begin{pmatrix} \bar f \\ 0 \end{pmatrix} \text{ for all }t>0, \quad W(0)=\begin{pmatrix}  y_0  \\ 0 \end{pmatrix},
	\end{align*}
	where $W(t)=\begin{pmatrix} \bar y(t) \\ \bar z(t) \end{pmatrix}.$ Since $(\Af_\nu +\Bf K)$ generates an analytic semigroup of negative type (see Theorem \ref{th:stb cnt}), $W(0)\in \Hf$ and $\begin{pmatrix} \bar f \\ 0 \end{pmatrix} \in L^2(0,\infty;\Hf)$, by \cite[Part II, Ch. 1, Proposition 3.1]{BDDM}, we have a unique strong solution $W \in L^2(0,\infty; \Hf)\cap C([0,\infty); \Hf])$ satisfying 
	\begin{align*}
		\|W\|_{L^2(0,\infty; \Hf)} + \|W\|_{L^\infty(0,\infty; \Hf)}  \le C (\| y_0\|_{L^2(\Omega)} + \|\bar f\|_{L^2(0,\infty;L^2(\Omega))}),
	\end{align*}
	for some $C>0.$ To prove the regularity, we set $p:=\bar y+\frac{\kappa}{\eta}\bar z,$ and then from \eqref{eq:GBHE_FbClsLinShift}, we obtain 
	\begin{equation*}
		\left\{
	\begin{aligned}
		& p_t -\eta \Delta p = (\nu+\frac{\kappa}{\eta}-\beta\gamma)\bar y + \frac{\kappa}{\eta} (\nu -\lambda)\bar z +\chi_{\mathcal{O}}(\bar y, \bar z) +\bar f\ \text{ in }\  \Omega \times (0,\infty),\\
		& p(x,t)=0 \ \text{ on }\  \Gamma\times (0,\infty), \\
		& p(x,0)= y_0(x) \ \text{ in }\ \Omega.
	\end{aligned}
	\right. 
	\end{equation*}
	Now, rest of the analysis follows by using similar arguments as used in Step 1 of Theorem \ref{th:localStrongsolGBHE}, and hence we skip here. 
\end{proof}

\noindent \textbf{Step 2.} One can easily show that $\wt{h}$ and $\wt{g}$ defined in \eqref{eqdef:wt-h & wt-g} also follow Propositions \ref{pps:GBHE-selfmap-I} - \ref{pps:GBHE-contrac-I}. 
\vskip 0.1cm 
\noindent \textbf{Step 3.} Finally, proceeding as in Step 2 of proof of Theorem \ref{th:localStrongsolGBHE}, we obtain the following stabilizability result.

\begin{Theorem} \label{th:stabcoupl-zero}
	Let $\nu \in (0,\nu_0)$ be arbitrary, where $\nu_0$ be as given in \eqref{eq:nu0}. There exists a linear continuous operator $K \in \mathcal{L}(\Hf, L^2(\Omega))$ and positive constants $\rho_2>0, $ and $\wt{M}>0,$ depending on $\alpha, \beta, \gamma, \lambda, \kappa, \delta$ such that for all $0<\rho \le \rho_2,$ and $y_0 \in H^1_0(\Omega)$ with 
	\begin{align*}
		\|y_0\|_{H^1_0(\Omega)}\le M \rho,
	\end{align*}
	the closed loop system \eqref{eq:GBHE-Lin-coup-shiftCls} admits a \emph{unique strong solution} $(\wt{y},\wt{z})$ satisfying 
	\begin{align*}
		\|\wt{y}\|_{L^2(0,\infty; H^2(\Omega))}^2 & +\|\wt{y}\|_{L^\infty(0,\infty; H^1_0(\Omega))}^{2}+\|\wt{y}\|_{H^1(0,\infty; L^2(\Omega))}^2 \\
		&+ \|\wt{y}\|_{L^{2(\delta+1)}(0,\infty; L^{6(\delta+1)}(\Omega))}^2+\|\wt{z}\|_{H^1(0,\infty; H^2(\Omega))}^2\le 2 \rho^2.
	\end{align*}
	Furthermore, we also have 
	\begin{align*}
		\|\wt{y}(\cdot,t)\|_{H^1_0(\Omega)}+\|\wt{z}(\cdot,t)\|_{H^2(\Omega)} \le C e^{-\nu t} \|y_0\|_{H^1_0(\Omega)}  \text{ for all } t>0,
	\end{align*}
	and for some $C>0.$
\end{Theorem}
Now, we write the stabilizability result for the integral equation \eqref{eq:GBHE-C0}. 

\begin{Theorem} \label{th:stabIneg-zero}
Let $\nu \in (0,\nu_0)$ be arbitrary, where $\nu_0$ be as given in \eqref{eq:nu0}. There exists a linear continuous operator $\wt{K} \in \mathcal{L}(L^2(0,\infty; L^2(\Omega)), L^2(\Omega))$ and positive constants $\rho_2>0, $ and $\wt{M}>0,$ depending on $\alpha, \beta, \gamma, \lambda, \kappa, \delta$ such that for all $0<\rho \le \rho_2,$ and $y_0 \in H^1_0(\Omega)$ with $
	\|y_0\|_{H^1_0(\Omega)} \le M \rho,$
the closed loop system 
\begin{equation}
	\left\{
	\begin{aligned}
		& y_t - \eta \Delta y +  \alpha y^\delta \sum_{i=1}^d \frac{\partial y}{\partial x_i} - \kappa \int_0^t e^{-\lambda (t-s)}\Delta y(\cdot,s)ds  = \beta y(1-y^\delta)(y^\delta-\gamma)+ \wt{K}y \ \text{ in }\  \Omega\times (0,\infty), \\
		& y=0  \ \text{ on }\  \Gamma\times (0,\infty), \\
		& y(x,0)=y_0 \ \text{ in }\ \Omega,
	\end{aligned}\right.
\end{equation}
admits a \emph{unique strong solution}
\begin{align*}
	y \in L^2(0,\infty; H^2(\Omega))\cap L^\infty(0,\infty; H^1_0(\Omega)) \cap H^1(0,\infty; L^2(\Omega)) \cap L^{2(\delta+1)}(0,\infty; L^{6(\delta+1)}(\Omega)),
\end{align*}
such that 
\begin{align*}
	&\|y(\cdot,t)\|_{H^1_0(\Omega)}\le C e^{-\nu t} \|y_0\|_{H^1_0(\Omega)}  \text{ for all } t>0, \text{ and } \\
	& \|e^{\nu \cdot} y\|_{L^2(0,\infty; H^2(\Omega))}^2 + \|e^{\nu \cdot} y\|_{H^1(0,\infty; L^2(\Omega))}^2 + \|e^{\nu \cdot} y\|_{L^{2(\delta+1)}(0,\infty; L^{6(\delta+1)}(\Omega))}^2 \le \rho^2,
\end{align*}
for some $C>0.$
\end{Theorem}

\section{Stabilizability around a non-constant steady state} \label{sec:stabaroundNonzero} In this section, we discuss the stabilizability of generalized Burgers-Huxley equation with memory around a non-constant steady state. Let us consider the following GBHE equation with memory:
{\small 
\begin{equation} \label{eq:GBHENC}
\left\{
    \begin{aligned}
& y_t - \eta \Delta y +  \alpha y^\delta \sum_{i=1}^d \frac{\partial y}{\partial x_i} - \kappa \int_0^t e^{-\lambda (t-s)}\Delta y(\cdot,s)ds + \beta y(y^\delta-1)(y^\delta-\gamma) =f_\infty + u\chi_{\mathcal{O}} \ \text{ in }\  \Omega\times (0,\infty), \\
& y=0  \ \text{ on }\  \Gamma\times (0,\infty), \\
        & y(x,0)=y_0 \ \text{ in }\ \Omega,
    \end{aligned}\right.
\end{equation}}
where $f_\infty \in L^2(\Omega)$ is a given stationary force term. 
Let $y_\infty$ be the solution of the following steady state equation:
\begin{equation} \label{eq:GBHE-ST}
\left\{
    \begin{aligned}
&  - \eta \Delta y_\infty +  \alpha y_\infty^\delta \sum_{i=1}^d \frac{\partial y_\infty}{\partial x_i}   + \beta y_\infty(y_\infty^\delta-1)(y_\infty^\delta-\gamma) - \frac{\kappa}{\lambda}\Delta y_\infty =f_\infty  \ \text{ in }\  \Omega, \\
& y_\infty=0  \ \text{ on }\  \Gamma.
    \end{aligned}\right.
\end{equation}
The existence of a strong solution of the above equation is studied in \cite{MTMAKRicGBHE}, and we have the following result: 

\begin{Theorem}[{\n \cite[Theorems 2.1, 2.3]{MTMAKRicGBHE}}] \label{th:Solv_GBHE-ST}
For given $f_\infty \in H^{-1}(\Omega),$ the system \eqref{eq:GBHE-ST} admits a weak solution $y_\infty \in H^1_0(\Omega) \cap L^{2(\delta+1)}(\Omega).$ Furthermore, if $ f_\infty \in L^2(\Omega),$ the solution $y_\infty \in H^2(\Omega)\cap H^1_0(\Omega) .$ 
\end{Theorem}

\begin{Remark}
1. For $d\in\{1,2,3\}$, we know that $H^2(\Omega)\cap H^1_0(\Omega)\subset L^p(\Omega)$ for any $p\in[2,\infty)$. 

2.	The uniqueness of a weak solution for the system \eqref{eq:GBHE-ST} follows  under appropriate conditions on $\nu,\kappa,\lambda,\alpha,\beta,\gamma$ and $f_{\infty}$ (\cite[Theorem 2.2, Remarks 1-3]{MTMAKRicGBHE}). As we are searching for  stabilizability around a non-constant steady state, we work with the solution $y_{\infty}$ obtained in Theorem \ref{th:Solv_GBHE-ST}. 
	\end{Remark}

As discussed in Section \ref{sec:contMethod}, establishing the stabilizability of $y$ around the steady state $ y_\infty $ reduces to demonstrating the stabilizability of $ y - y_\infty $ around zero. Therefore, it is natural to introduce the transformation $ w := y - y_\infty $ and analyze the stabilizability of $ w $ around zero, where $ w $ satisfies the following system:
\begin{equation}  \label{eq:GBHE-Lin-w-y_inft}
	\left\{
	\begin{aligned}
		&  w_t -\eta \Delta w  + \alpha \left( (w+y_\infty)^\delta \nabla (w+y_\infty) \cdot \boldsymbol{1} - y_\infty^\delta \nabla y_\infty\cdot \boldsymbol{1} \right)  - \kappa \int_0^t e^{-\lambda(t-s)}\Delta w(\cdot,s) ds \\
		& \quad  + \beta \left( (w+y_\infty) ((w+y_\infty)^\delta -1) ((w+y_\infty)^\delta -\gamma) - y_\infty (y_\infty^\delta -1) (y_\infty^\delta -\gamma)  \right)\\&\quad+ \frac{\kappa}{\lambda}e^{-\lambda t}\Delta y_\infty=u \chi_{\mathcal{O}},\\
		&   w(x,t)=0, \ \text{ on }\   \Gamma\times (0,\infty),\\
		&w(x,0)=y_0(x)-y_\infty(x), \, \text{ for all } \, x\, \ \text{ in }\  \, \Omega.
	\end{aligned}
	\right. 
\end{equation}
Let us first consider the following linear controlled equation:
\begin{equation}  
	\left\{
	\begin{aligned}
		&  w_t -\eta \Delta w +\beta\gamma w  - \kappa \int_0^t e^{-\lambda(t-s)}\Delta w(\cdot,s) ds =u \chi_{\mathcal{O}} \ \text{ in }\   \Omega\times (0,\infty),\\
		&   w(x,t)=0, \text{ for all } \,(x,t)\in \Gamma\times (0,\infty),\\
		&w(x,0)=y_0(x)-y_\infty(x)=:w_0, \, \text{ for all } \, x \, \ \text{ in }\  \, \Omega.
	\end{aligned}
	\right. 
\end{equation}
Then the stabilizability of this heat equation with memory is already established in Theorem \ref{th:stb cnt}. Now, we aim to show that the feedback operator $K$ as in Remark \ref{rem:feedbackop} stabilizes the system \eqref{eq:GBHE-Lin-w-y_inft}. To do so, let us first re-write it in a coupled form as
\begin{equation}  \label{eq:LinearizedCouple-NC}
	\left\{
	\begin{aligned}
		&  w_t -\eta \Delta w  + \alpha \left( (w+y_\infty)^\delta \nabla (w+y_\infty) \cdot \boldsymbol{1} - y_\infty^\delta \nabla y_\infty\cdot \boldsymbol{1}\right)  - \kappa\Delta v \\
		& \quad  + \beta \left( (w+y_\infty) ((w+y_\infty)^\delta -1) ((w+y_\infty)^\delta -\gamma) - y_\infty (y_\infty^\delta -1) (y_\infty^\delta -\gamma)  \right)\\&\quad+ \frac{\kappa}{\lambda}e^{-\lambda t}\Delta y_\infty=u \chi_{\mathcal{O}},\\
		& v_t+\lambda v - w= 0 \ \text{ in }\  \Omega \times (0,\infty), \\
		&   w(x,t)=0, \quad v(x,t)= 0 \,  \text{ for all } \,(x,t)\in \Gamma\times (0,\infty),\\
		&w(x,0)=w_0(x), \quad v(x,0)=0 \text{ for all } x\in \Omega.
	\end{aligned}
	\right.
\end{equation}
Since, our aim is to obtain stabilizability with decay $\nu$ for any $\nu \in (0,\lambda),$ we now consider the following  shifted closed loop system:
\begin{equation}  \label{eq:GBHE-NLClsCoupShifNCST} 
	\left\{
	\begin{aligned}
		&  \wt{w}_t -\eta \Delta \wt{w} -\nu \wt{w}  - \kappa\Delta \wt{v}  \\&= \chi_{\mathcal{O}} K(\wt{w}, \wt{v}) - \alpha \left( e^{-\nu \delta t} (\wt{w}+e^{\nu  t}y_\infty)^\delta \nabla (\wt{w}+e^{\nu t}y_\infty) \cdot \boldsymbol{1} - e^{\nu t}y_\infty^\delta \nabla y_\infty\cdot \boldsymbol{1}\right)  \\
		& \quad  - \beta \Big( (\wt{w}+ e^{\nu t}y_\infty) (e^{-\nu \delta t}(\wt{w} + e^{\nu t}y_\infty)^\delta -1) (e^{-\nu \delta t}(\wt{w}+ e^{\nu  t}y_\infty)^\delta -\gamma) \\
		& \quad  - e^{\nu t}y_\infty (y_\infty^\delta -1) (y_\infty^\delta -\gamma)  \Big) - \frac{\kappa}{\lambda}e^{(\nu-\lambda) t}\Delta y_\infty \ \text{ in }\  \Omega \times (0,\infty),\\
		& \wt{v}_t+\lambda \wt{v} -\nu \wt{v} - \wt{w}= 0 \ \text{ in }\  \Omega \times (0,\infty), \\
		&   \wt{w}(x,t)=0, \quad \wt{v}(x,t)= 0 \,  \text{ for all } \,(x,t)\in \Gamma\times (0,\infty),\\
		& \wt{w}(x,0)=w_0(x), \quad \wt{v}(x,0)=0 \text{ for all } x\in \Omega,
	\end{aligned}
	\right.
\end{equation}
where $\wt{w}(\cdot,t)=e^{\nu t}w(\cdot,t),$  $\wt{v}(\cdot,t)=e^{\nu t}v(\cdot,t),$ and $K$ is the same as in Remark \ref{rem:feedbackop}. Let us define
\begin{align} \label{eq:h1}
	h_1(\wt{w}):= \alpha \left( e^{-\nu \delta t} (\wt{w}+e^{\nu  t}y_\infty)^\delta \nabla (\wt{w}+e^{\nu t}y_\infty) \cdot \boldsymbol{1}  - e^{\nu t}y_\infty^\delta \nabla y_\infty\cdot \boldsymbol{1}	\right),
\end{align}
and 
\begin{equation} \label{eq:g1}
\begin{aligned} 
	g_1(\wt{w}):= & 
	\Big( \beta e^{-2\nu \delta t} (\wt{w}+ e^{\nu t}y_\infty)^{2\delta+1} -\beta (1+\gamma) e^{-\nu \delta t} (\wt{w}+ e^{\nu t}y_\infty)^{\delta+1} \Big) \\
	& \qquad - \Big( \beta e^{\nu t}y_\infty^{2\delta+1} -\beta (1+\gamma) e^{\nu t}y_\infty^{\delta+1} \Big).
\end{aligned}
\end{equation}
Thus, we obtain the following system:
\begin{equation}  \label{eq:ff-nlclslpsys}
	\left\{
	\begin{aligned}
		&  \wt{w}_t -\eta \Delta \wt{w} +(\beta\gamma-\nu) \wt{w}  - \kappa\Delta \wt{v}  = \chi_{\mathcal{O}} K(\wt{w}, \wt{v}) - h_1(\wt{w}) - g_1(\wt{w}) - \frac{\kappa}{\lambda}e^{(\nu-\lambda) t}\Delta y_\infty \ \text{ in }\  \Omega \times (0,\infty),\\
		& \wt{v}_t+\lambda \wt{v} -\nu \wt{v} - \wt{w}= 0 \ \text{ in }\  \Omega \times (0,\infty), \\
		&   \wt{w}(x,t)=0, \quad \wt{v}(x,t)= 0 \,  \text{ for all } \,(x,t)\in \Gamma\times (0,\infty),\\
		& \wt{w}(x,0)=w_0(x), \quad \wt{v}(x,0)=0 \text{ for all } x\in \Omega.
	\end{aligned}
	\right. 
\end{equation}



It will be enough to show the existence of a stable solution of the closed loop non-linear system \eqref{eq:GBHE-NLClsCoupShifNCST}, and we perform that using the Banach fixed point theorem. Before stating and proving our main result of this section, we first show that the non-linear terms in the first equation in \eqref{eq:GBHE-NLClsCoupShifNCST} belongs to the space $L^2(0,\infty;L^2(\Omega))$ and prove the required results to establish the self and contraction maps. 

\medskip
Recall the spaces $D,$ $D_\rho,$ $B,$ and $B_\rho$ defined in \eqref{eq:D} and \eqref{eq:D_rho - B_rho}. Also, recall the non-linear terms defined by $h_1$ and $g_1$ from \eqref{eq:h1} and \eqref{eq:g1}. 

\begin{Proposition} \label{pps:h_1 g_1 L2}
Let $y_\infty \in H^2(\Omega)\cap H^1_0(\Omega)$ be a strong solution of the problem \eqref{eq:GBHE-ST} as discussed in Theorem \ref{th:Solv_GBHE-ST}, and let $D$ be as defined in \eqref{eq:D}. Then for any $\wt{w} \in D,$ the functions $h_1$ and $g_1$ defined in \eqref{eq:h1} and \eqref{eq:g1}, respectively, satisfy
\begin{align}
&\|h_1(\wt{w})\|_{L^2(0,\infty; L^2(\Omega))} \le M_4 \left(\|\wt{w}\|_D^{\delta+1} + \|y_\infty\|_{H^2(\Omega)} \|\wt{w}\|_D^\delta + \|y_\infty\|_{H^2(\Omega)}^{\delta}\|\wt{w}\|_D+ \|y_\infty\|_{L^{3\delta}(\Omega)}^{\delta}\|\wt{w}\|_D\right),
\label{eqn-f1}\\&\|g_1(\wt{w})\|_{L^2(0,\infty; L^2(\Omega))} \le M_4 \left( \|\wt{w}\|_D^{2\delta+1} + \|y_\infty\|_{L^{4\delta}(\Omega)}^{2\delta}\|\wt{w}\|_D+\|\wt{w}\|_D^{\delta+1}+\|y_\infty\|_{L^{2\delta}(\Omega)}^\delta \|\wt{w}\|_D \right),\label{eqn-f2}
\end{align}
for some constant $M_4>0.$
\end{Proposition}

\begin{proof}
Let us first prove \eqref{eqn-f1}.  Note that for any $\wt{w} \in D,$ we have 
\begin{align*}
h_1(\wt{w}) & =\alpha \left( e^{-\nu \delta t} (\wt{w}+e^{\nu  t}y_\infty)^\delta \nabla (\wt{w}+e^{\nu t}y_\infty) \cdot \boldsymbol{1} - e^{\nu t}y_\infty^\delta \nabla y_\infty\cdot \boldsymbol{1}\right)	& \\
& =  \alpha \left( e^{-\nu \delta t} (\wt{w}+e^{\nu  t}y_\infty)^\delta \nabla \wt{w} \cdot \boldsymbol{1} \right) + \alpha \left( e^{-\nu \delta t} (\wt{w}+e^{\nu  t}y_\infty)^\delta e^{\nu t}\nabla y_\infty \cdot \boldsymbol{1} - e^{\nu t}y_\infty^\delta \nabla y_\infty\cdot \boldsymbol{1}\right).
\end{align*}
A use of generalized H\"{o}lder's inequality leads to 
\begin{align*}
 \| e^{-\nu \delta t} (\wt{w} & +e^{\nu  t}y_\infty)^\delta \nabla \wt{w} \cdot \boldsymbol{1}\|_{L^2(0,\infty; L^2(\Omega))}^2 \\
 & = \int_0^\infty \left\| e^{-\nu \delta t}  (\wt{w}(t)+e^{\nu  t}y_\infty)^\delta \nabla \wt{w}(t) \cdot \boldsymbol{1} \right\|^2 dt \\
 & \le C \left( \int_0^\infty \| \wt{w}^\delta(t) \nabla \wt{w}(t) \cdot \boldsymbol{1} \|^2 dt + \int_0^\infty \| y_\infty^\delta \nabla \wt{w}(t) \cdot \boldsymbol{1} \|^2 dt  \right) \\
 & \le C\|\wt{w}\|_{L^\infty(0,\infty; L^{3\delta}(\Omega))}^{2\delta} \|\wt{w}\|_{L^2(0,\infty; H^2(\Omega))}^2 +C \|y_\infty\|_{L^{3\delta}(\Omega)}^{2\delta} \|\wt{w}\|_{L^2(0,\infty; H^2(\Omega))}^2.
\end{align*}
Now, let us set the function $\varphi(w)=(e^{-\nu t}w+y_\infty)^{\delta}$. Then, we have 
\begin{align}\label{aux_res2}
 (e^{-\nu  t}\wt{w}(t)+y_\infty)^\delta - y_\infty^\delta&=\varphi(\wt{w})-\varphi(0)=\int_0^1\frac{d}{d\theta}\varphi(\theta\wt{w})d\theta=\int_0^1\varphi'(\theta\wt{w})\wt{w} \,d\theta\no\\
&=\delta\int_0^1(\theta e^{-\nu t}\wt{w}+y_\infty)^{\delta-1}\wt{w} \, d\theta.
\end{align}
We estimate the second term using \eqref{aux_res2} as follows:
\begin{align*}
 \| e^{-\nu \delta t} & (\wt{w}+e^{\nu  t}y_\infty)^\delta e^{\nu t}\nabla y_\infty \cdot \boldsymbol{1} - e^{\nu t}y_\infty^\delta \nabla y_\infty\cdot \boldsymbol{1} \|_{L^2(0,\infty; L^2(\Omega))}^2  \\
 & = \int_0^\infty    \left\| \left( (e^{-\nu  t}\wt{w}(t)+y_\infty)^\delta - y_\infty^\delta \right) e^{\nu t} \nabla y_\infty\cdot \boldsymbol{1} \right\|^2 dt \\
 & =  \int_0^\infty    \left\|  e^{-\nu\delta t}\wt{w}(t)\left(\int_0^1(\theta\wt{w}(t)+ e^{\nu t}y_\infty)^{\delta-1} d\theta\right)  \nabla y_\infty\cdot \boldsymbol{1} \right\|^2 dt \\
& \le C \int_0^\infty   \left(  \| e^{-\nu\delta t}\wt{w}(t)|\wt{w}(t)|^{\delta-1} \nabla y_\infty\cdot \boldsymbol{1} \|^2 +  \|e^{-\nu t} \wt{w}(t)|y_\infty|^{\delta-1} \nabla y_\infty\cdot \boldsymbol{1} \|^2 \right) dt \\
&  \le C\|y_\infty\|_{H^2(\Omega)}^2 \|\wt{w}\|_{L^{2\delta}(0,\infty; L^{3\delta}(\Omega))}^{2\delta} + C \|y_\infty\|_{H^2(\Omega)}^{2\delta} \|\wt{w}\|_{L^2(0,\infty;H^1(\Omega))}^2,
\end{align*}
which concludes the proof of \eqref{eqn-f1}.

\medskip 
\noindent Let us now prove \eqref{eqn-f2}. First, we observe that 
\begin{align*}
    \frac{g_1(\wt{w})}{\beta} & =  e^{-2\nu \delta t}\left( (\wt{w}+ e^{\nu t}y_\infty)^{2\delta+1} - (e^{\nu  t} y_\infty)^{2\delta+1}\right) -(1+\gamma) e^{-\nu \delta t}\left(  (\wt{w}+ e^{\nu t}y_\infty)^{\delta+1} - (e^{\nu  t} y_\infty)^{\delta+1} \right) .
\end{align*}
The first term can be estimated using a similiar application of \eqref{aux_res2}  as follows:
\begin{align*}
\Big\| e^{-2\nu \delta t} & \Big((\wt{w}+ e^{\nu t}y_\infty)^{2\delta+1} - (e^{\nu  t} y_\infty)^{2\delta+1} \Big)\Big\|_{L^2(0,\infty;L^2(\Omega))}^2 \\
& =  \int_0^\infty \left\|\wt{w}(t) e^{-2\nu \delta t} \left(\int_0^1(\theta\wt{w}(t)+ e^{\nu t}y_\infty)^{2\delta} d\theta\right) \right\|^2dt \\
& \le C \int_0^\infty \|\wt{w}^{\delta+1}(t)\|_{L^6(\Omega)}^2 \|\wt{w}^\delta(t)\|_{L^3(\Omega)}^2 dt + C\int_0^\infty \|\wt{w}(t)\|_{L^\infty(\Omega)}^2\|y_\infty^{2\delta}\|^2 dt \\
 & \le C\|\wt{w}\|_{L^\infty(0,\infty; L^{3\delta}(\Omega))}^{2\delta}  \|\wt{w}\|_{L^{2(\delta+1)}(0,\infty; L^{6(\delta+1)}.(\Omega))}^{2(\delta+1)} + \|y_\infty\|_{L^{4\delta}(\Omega)}^{4\delta}\|\wt{w}\|_{L^2(0,\infty;H^2(\Omega))}^2.
\end{align*}
We estimate the second term in a similar way as
\begin{align*}
\Big\| e^{-\nu \delta t}\Big( (\wt{w} & + e^{\nu t}y_\infty)^{\delta+1} - (e^{\nu  t} y_\infty)^{\delta+1} \Big)\Big\|_{L^2(0,\infty;L^2(\Omega))}^2 \\
& = \int_0^\infty\left \|e^{-\nu \delta t}\wt{w}(t) \left(\int_0^1(\theta\wt{w}(t)+ e^{\nu t}y_\infty)^{\delta} d\theta\right)\right\|^2 dt   \\
& \le C \int_0^\infty \|\wt{w}(t)^{\delta +1}\|^2dt +\int_0^\infty \|\wt{w}(t)y_\infty^\delta\|^2dt\\
& \le \|\wt{w}\|_{L^{2(\delta+1)}(0,\infty; L^{2(\delta+1)}(\Omega))}^{2(\delta+1)} + \|y_\infty\|_{L^{2\delta}(\Omega)}^{2\delta}\|\wt{w}\|_{L^2(0,\infty;H^2(\Omega))}^2.
\end{align*}
Combining these estimates, we conclude the proof of the estimate \eqref{eqn-f2}.
\end{proof}

\begin{Proposition} \label{pps-h1g1-lip}
Let $y_\infty \in H^2(\Omega)\cap H^1_0(\Omega)$ be a strong solution of the problem \eqref{eq:GBHE-ST} as discussed in Theorem \ref{th:Solv_GBHE-ST}, and let $D$ be as defined in \eqref{eq:D}. Then for any $\wt{w}_1, \wt{w}_2 \in D,$ the functions $h_1$ and $g_1$ defined in \eqref{eq:h1} and \eqref{eq:g1}, respectively, satisfy
\begin{align*}
	(a)  \,   \|h_1(\wt{w}_1) - h_1(\wt{w}_2)\|_{L^2(0,\infty; L^2(\Omega))} &\le M_5 \|\wt{w}_1 -\wt{w}_2\|_D \Big( \|\wt{w}_1 \|_D^\delta+\|\wt{w}_2 \|_D^\delta +\|\wt{w}_1\|_D\|\wt{w}_2\|_D^{\delta-1} \\
	& \qquad + (\|\wt{w}_1\|_D^{\delta-1} + \|\wt{w}_2\|_D^{\delta-1})\|y_\infty\|_{H^2(\Omega)} +\|y_\infty\|_{H^2(\Omega)}^\delta \\
	& \qquad + \|\wt{w}_1\|_D \|y_\infty\|_{L^{6(\delta-1)}(\Omega)}^{\delta-1} + \|y_\infty\|_{L^{3\delta}(\Omega)}^{\delta}   \Big), 
\end{align*}
and 
\begin{align*}
	(b)  \,	\|g_1(\wt{w}_1) - g_1(\wt{w}_2)\|_{L^2(0,\infty; L^2(\Omega))} & \le M_5 \|\wt{w}_1-\wt{w}_2\|_D\Big( \|\wt{w}_1\|_D^{2\delta}+ \|\wt{w}_2\|_D^{2\delta} + \|y_\infty\|_{L^{6\delta}(\Omega)}^{2\delta} \\
	& \qquad  + \|\wt{w}_1\|_D^{\delta}+ \|\wt{w}_2\|_D^{\delta} + \|y_\infty\|_{H^2(\Omega)}^{\delta}\Big),
\end{align*}
for some constant $M_5>0.$
\end{Proposition}

\begin{proof}
(a) We start with an observation that
\begin{align*}
 h_1(\wt{w}_1) - h_1(\wt{w}_2) & =    \alpha e^{-\nu \delta t} \left(  (\wt{w}_1+e^{\nu  t}y_\infty)^\delta \nabla \wt{w}_1 \cdot \boldsymbol{1} -   (\wt{w}_2+e^{\nu  t}y_\infty)^\delta \nabla \wt{w}_2 \cdot \boldsymbol{1} \right) \\
 & \quad + \alpha e^{-\nu \delta t}\left((\wt{w}_1+e^{\nu  t}y_\infty)^\delta -  (\wt{w}_2+e^{\nu  t}y_\infty)^\delta\right)  e^{\nu t}\nabla y_\infty \cdot \boldsymbol{1}\\
 & =    \alpha e^{-\nu \delta t} \left( \left( (\wt{w}_1+e^{\nu  t}y_\infty)^\delta -   (\wt{w}_2+e^{\nu  t}y_\infty)^\delta \right)\nabla \wt{w}_1 \cdot \boldsymbol{1} \right) \\
 & \quad + \alpha e^{-\nu \delta t} \left(  (\wt{w}_2+e^{\nu  t}y_\infty)^\delta \nabla (\wt{w}_1 -\wt{w}_2) \cdot \boldsymbol{1}  \right) \\
 & \quad + \alpha e^{-\nu \delta t}\left((\wt{w}_1+e^{\nu  t}y_\infty)^\delta -  (\wt{w}_2+e^{\nu  t}y_\infty)^\delta\right)  e^{\nu t}\nabla y_\infty \cdot \boldsymbol{1}.
\end{align*}
To estimate the first term, a repeated application of \eqref{aux_res1} and Cauchy-Schwarz inequality leads to
\begin{align}\label{1st}
 &\| e^{-\nu \delta t} \big( (\wt{w}_1  +e^{\nu  t}y_\infty)^\delta   -   (\wt{w}_2+e^{\nu  t}y_\infty)^\delta \big)\nabla \wt{w}_1 \cdot \boldsymbol{1} \|^2_{L^2(0,\infty; L^2(\Omega))} \no\\
 & =  \int_0^\infty \big\|e^{-2\nu \delta t}(\wt{w}_1 - \wt{w}_2)(t) \Big(\int_0^1\big( \theta(\wt{w}_1(t)+ e^{\nu t}y_\infty)+ (1-\theta)(\wt{w}_2(t)+ e^{\nu t}y_\infty) \big)^{\delta-1} d\theta\Big) \nabla \wt{w}_1(t) \cdot \boldsymbol{1}\big\|^2 dt  \no \\
 & \le C \int_0^\infty  e^{-2\nu \delta t} \|\wt{w}_1(t) - \wt{w}_2(t)\|_{L^6(\Omega)}^2 \||\wt{w}_1(t)|^{\delta-1} + |\wt{w}_2(t)|^{\delta-1} +e^{(\delta-1)\nu t} |y_\infty|^{\delta-1}\|^2_{L^6(\Omega)}\|\nabla \wt{w}_1(t)\|_{L^6(\Omega)}^2 dt\no\\
 & \le C \|\wt{w}_1-\wt{w}_2\|_{L^{\infty}(0,\infty; H^1_0(\Omega))}^2 \Big( \|\wt{w}_1\|_{L^{\infty}(0,\infty; H^1_0(\Omega))}^{2\delta-2} + \|\wt{w}_2\|_{L^{\infty}(0,\infty; H^1_0(\Omega))}^{2\delta-2} \no \\
 & \qquad + \|y_\infty\|_{L^{6(\delta-1)}(\Omega)}^{2\delta-2} \Big) \|\wt{w}_1\|_{L^2(0,\infty; H^2(\Omega))}^2.
\end{align}
For the second term, we follow an analogous path
\begin{align}\label{2nd}
\|e^{-\nu \delta t} & \left(  (\wt{w}_2+e^{\nu  t}y_\infty)^\delta \nabla (\wt{w}_1 -\wt{w}_2) \cdot \boldsymbol{1}  \right)\|^2_{L^2(0,\infty; L^2(\Omega))} \no\\
& = \int_0^\infty \| (e^{-\nu  t}\wt{w}_2(t)+y_\infty)^\delta \nabla (\wt{w}_1 -\wt{w}_2)(t) \cdot \boldsymbol{1}  \|^2 dt \no\\
& \le C\int_0^\infty \| e^{-\nu  t}\wt{w}_2^\delta(t) \nabla (\wt{w}_1 -\wt{w}_2)(t) \cdot \boldsymbol{1}  \|^2 dt + C\int_0^\infty \| y_\infty^\delta \nabla (\wt{w}_1 -\wt{w}_2)(t) \cdot \boldsymbol{1}  \|^2 dt\no \\
& \le C  \int_0^\infty  \| \wt{w}_2^\delta(t)\|_{L^3(\Omega)}^2 \|\nabla (\wt{w}_1(t) -  \wt{w}_2(t))\|_{L^6(\Omega)}^2 dt \no\\
& \qquad +  C  \int_0^\infty  \| y_\infty^\delta\|_{L^3(\Omega)}^2 \|\nabla (\wt{w}_1(t) -  \wt{w}_2(t))\|_{L^6(\Omega)}^2 dt\no\\
 & \le C \left(  \|\wt{w}_2\|_{L^{\infty}(0,\infty; L^{3\delta}(\Omega))}^{2\delta} + \| y_\infty\|_{L^{3\delta}(\Omega)}^{2\delta} \right)\|\wt{w}_1-\wt{w}_2\|_{L^2(0,\infty; H^2(\Omega))}^2.
\end{align}
One can estimate the last term in a similar manner as
\begin{align}\label{3rd}
	&\| e^{-\nu \delta t} \big(  (\wt{w}_1+e^{\nu  t}y_\infty)^\delta   -   (\wt{w}_2+e^{\nu  t}y_\infty)^\delta \big) e^{\nu t}\nabla y_\infty \cdot \boldsymbol{1} \|^2_{L^2(0,\infty; L^2(\Omega))} \no\\
	& =  \int_0^\infty \left\| e^{\nu(1- \delta) t}(\wt{w}_1 - \wt{w}_2)(t) \left(\int_0^1( \theta\wt{w}_1(t)+  (1-\theta)\wt{w}_2(t)+e^{\nu t}y_\infty)^{\delta-1}  d\theta\right) \nabla y_\infty \cdot \boldsymbol{1}\right\|^2 dt \no  \\
	& \le C \int_0^\infty  e^{\nu(1- \delta) t} \|\wt{w}_1(t) - \wt{w}_2(t)\|_{L^6(\Omega)}^2 \||\wt{w}_1(t)|^{\delta-1} + |\wt{w}_2(t)|^{\delta-1} + |e^{\nu t}y_\infty|^{\delta-1}\|^2_{L^6(\Omega)}\|\nabla y_\infty\|_{L^6(\Omega)}^2 dt \no\\
	& \le C \|\wt{w}_1-\wt{w}_2\|_{L^2(0,\infty; H^1_0(\Omega))}^2 \Big( \|\wt{w}_1\|_{L^{\infty}(0,\infty; H^1_0(\Omega))}^{2\delta-2} + \|\wt{w}_2\|_{L^{\infty}(0,\infty; H^1_0(\Omega))}^{2\delta-2} \no \\
	& \qquad + \|y_\infty\|_{L^{6(\delta-1)}(\Omega)}^{2\delta-2} \Big) \|y_\infty\|_{H^2(\Omega)}^2.
 \end{align}
Adding the inequalities obtained in \eqref{1st}, \eqref{2nd} and \eqref{3rd}, we prove (a).

\vskip 0.1cm
\noindent (b) Note that 
\begin{align*}
\frac{1}{\beta} \left( g_1(\wt{w}_1) - g_1(\wt{w}_2) \right) & =  e^{-2\nu \delta t}\left( (\wt{w}_1+ e^{\nu t}y_\infty)^{2\delta+1} - (\wt{w}_2+ e^{\nu t}y_\infty)^{2\delta+1}\right) \\
    & \qquad -(1+\gamma) e^{-\nu \delta t}\left(  (\wt{w}_1+ e^{\nu t}y_\infty)^{\delta+1} - (\wt{w}_2+ e^{\nu t}y_\infty)^{\delta+1} \right).
 \end{align*}
Again using \eqref{aux_res1}, the first term can be estimated as 
\begin{align*}
	\Big\| e^{-2\nu \delta t} & \Big( (\wt{w}_1+ e^{\nu t}y_\infty)^{2\delta+1}  - (\wt{w}_2+ e^{\nu t}y_\infty)^{2\delta+1}\Big)\Big\|^2_{L^2(0,\infty; L^2(\Omega))} \\
	& =  \int_0^\infty \Big\| e^{-2\nu \delta t}(\wt{w}_1-\wt{w}_2)(t) \big(\int_0^1(\theta \wt{w}_1(t)+ (1-\theta) \wt{w}_2(t)+ e^{\nu t}y_\infty)^{2\delta} d\theta\big)\Big\|^2 dt \\
	& \le C \|\wt{w}_1-\wt{w}_2\|_{L^\infty(0,\infty; H^1_0(\Omega))}^2 \left( \int_0^\infty \|\wt{w}_1(t)\|_{L^{6\delta}(\Omega)}^{4\delta} dt + \int_0^\infty \|\wt{w}_2(t)\|_{L^{6\delta}(\Omega)}^{4\delta} dt  \right) \\
	& \qquad + C \|y_\infty\|_{L^{6\delta}(\Omega)}^{4\delta} \|\wt{w}_1-\wt{w}_2\|_{L^2(0,\infty; H^1_0(\Omega))}^2.
\end{align*}
For the case $d=3$, $\delta=1,$ or $d=2,$ the above estimate is enough. However for $\delta= 2$ and $d=3,$  we use interpolation to have
\begin{align*}
	\int_0^\infty  \|\wt{w}_i(t)\|_{ L^{6\delta}(\Omega)}^{4\delta} dt & \le C \int_0^\infty \|\wt{w}_i(t)\|_{ L^{6(\delta-1)}(\Omega)}^{2\delta} \|\wt{w}_i(t)\|_{ L^{6(\delta+1)}(\Omega)}^{2\delta} \\
	& \le C \int_0^\infty \|\wt{w}_i(t)\|_{ L^{6(\delta-1)}(\Omega)}^{2(\delta-1)} \|\wt{w}_i(t)\|_{ L^{6(\delta+1)}(\Omega)}^{2(\delta+1)} \\
	&\le C_s \|\wt{w}_i\|_{L^\infty (0,\infty; H^1_0(\Omega))}^{2(\delta-1)} \|\wt{w}_i\|_{L^{2(\delta+1)} (0,\infty; L^{6(\delta+1)}(\Omega))}^{2(\delta+1)} \quad i=1,2.
\end{align*}
Now to estimate the second term,  we proceed similarly using \eqref{aux_res1} and obtain 
\begin{align*}
	\Big\| e^{-2\nu \delta t} & \Big( (\wt{w}_1+ e^{\nu t}y_\infty)^{\delta+1}  - (\wt{w}_2+ e^{\nu t}y_\infty)^{\delta+1}\Big)\Big\|^2_{L^2(0,\infty; L^2(\Omega))} \\
	& =  \int_0^\infty \left\|(\wt{w}_1-\wt{w}_2)(t) \int_0^1(\theta\wt{w}_1(t)+ (1-\theta) \wt{w}_2(t)+e^{\nu t}y_\infty)^{\delta} d\theta \right\|^2 dt \\
	& \le C \|\wt{w}_1-\wt{w}_2\|_{L^\infty(0,\infty; H^1_0(\Omega))}^2 \left( \int_0^\infty \|\wt{w}_1(t)\|_{L^{3\delta}(\Omega)}^{2\delta} dt + \int_0^\infty \|\wt{w}_2(t)\|_{L^{3\delta}(\Omega)}^{2\delta} dt  \right) \\
	& \qquad + C \|y_\infty\|_{L^{6\delta}(\Omega)}^{2\delta} \|\wt{w}_1-\wt{w}_2\|_{L^2(0,\infty; H^1_0(\Omega))}^2.
\end{align*}
Combining these estimates, we conclude the required result.
\end{proof}

Now, we are ready to state and prove the main stabilization result of the non-linear system around non-constant steady state $y_\infty$ which is a solution of stationary generalized Burgers-Huxley equation. 

\begin{Theorem} \label{th:stabIneg-NC}
Suppose that $\nu \in (0,\lambda)$ be arbitrary, and $w_0\in H^1_0(\Omega)$. Also, suppose that $y_\infty \in H^2(\Omega)\cap H^1_0(\Omega)$ is a solution of the problem \eqref{eq:GBHE-ST}.  Then there exists a feedback operator $K \in \mathcal{L}(L^2(\Omega),\Hf) ,$ $\rho_3>0$ and $\overline{M} >0,$ such that for all $0<\rho<\rho_3,$  satisfying 
\begin{align*}
	\|y_\infty\|_{H^2(\Omega)}\le \overline{M} \rho \text{ and } \|w_0\|_{H^1_0(\Omega)}\le \overline{M} \rho,
\end{align*}
the closed loop system \eqref{eq:ff-nlclslpsys} admits a unique solution 
\begin{align*}
& \wt{w}	\in L^2(0,\infty; H^2(\Omega))\cap L^\infty(0,\infty; H^1_0(\Omega)) \cap L^{2(\delta+1)}(0,\infty; L^{6(\delta+1)}(\Omega)) \cap H^1(0,\infty; L^2(\Omega) ), \\
&  \wt{v} \in H^1(0,\infty; H^2(\Omega)\cap H^1_0(\Omega)),
\end{align*} 
such that 
\begin{align*}
	\|\wt{w}\|_{L^2(0,\infty; H^2(\Omega))}^2 & +\|\wt{w}\|_{L^\infty(0,\infty; H^1_0(\Omega))}^{2}+\|\wt{w}\|_{H^1(0,\infty; L^2(\Omega))}^2 \\
	&+ \|\wt{w}\|_{L^{2(\delta+1)}(0,\infty; L^{6(\delta+1)}(\Omega))}^{2}+\|\wt{v}\|_{H^1(0,\infty; H^2(\Omega))}^2\le 2 \rho^2.
\end{align*}
Furthermore, we also have 
\begin{align*}
	\|\wt{w}(\cdot,t)\|_{H^1_0(\Omega)}+\|\wt{v}(\cdot,t)\|_{H^2(\Omega)} \le C e^{-\nu t} \left( \|y_\infty\|_{H^2(\Omega)}+ \|w_0\|_{H^1_0(\Omega)} \right) \text{ for all } t>0,
\end{align*}
and for some $C>0.$
\end{Theorem}

\begin{proof}
To prove this theorem, we use the Banach fixed point theorem. For a given $\psi\in D,$ we consider the system
\begin{equation}  \label{eq:ff-nh-nlclslpsys}
	\left\{
	\begin{aligned}
	&  \wt{w}_{\psi_t} -\eta \Delta \wt{w}_\psi +(\beta\gamma -\nu) \wt{w}_\psi  - \kappa\Delta \wt{v}_\psi  \\&\quad= \chi_{\mathcal{O}} K(\wt{w}_\psi, \wt{v}_\psi) - h_1(\psi) - g_1(\psi) - \frac{\kappa}{\lambda}e^{(\nu-\lambda) t}\Delta y_\infty \ \text{ in }\  \Omega \times (0,\infty),\\
	& \wt{v}_{\psi_t}+\lambda \wt{v}_\psi -\nu \wt{v}_\psi - \wt{w}_\psi= 0 \ \text{ in }\  \Omega \times (0,\infty), \\
	&   \wt{w}_\psi(x,t)=0, \quad \wt{v}_\psi(x,t)= 0 \,  \text{ for all } \,(x,t)\in \Gamma\times (0,\infty),\\
	& \wt{w}_\psi(x,0)=w_0, \quad \wt{v}_\psi(x,0)=0 \text{ for all } x\in \Omega,
	\end{aligned}
	\right.
\end{equation}
where $h_1$ and $g_1$ are as defined in \eqref{eq:h1} and \eqref{eq:g1}, respectively, and $K$ is as in Remark \ref{rem:feedbackop}. We first show that for any given $\psi \in D,$ the above system has a unique strong solution
\begin{align*}
	& \wt{w}_\psi \in  L^2(0,\infty; H^2(\Omega)) \cap L^\infty(0,\infty; H^1_0(\Omega)) \cap H^1(0,\infty; L^2(\Omega)) \cap L^{2(\delta+1)}(0,\infty; L^{6(\delta+1)}(\Omega)),\\
	& \wt{v}_\psi \in H^1(0,\infty; H^2(\Omega)\cap H^1_0(\Omega)).
\end{align*}
Indeed, this follows from the fact that $h_1(\psi), g_1(\psi)\in L^2(0,\infty; L^2(\Omega))$ (see Proposition \ref{pps:h_1 g_1 L2}), $e^{(\nu-\lambda) t}\Delta y_\infty \in L^2(0,\infty; L^2(\Omega))$ as $\nu<\lambda,$ and then using Theorem \ref{th:reg-nh-ff-Psi}. Moreover, we have
\begin{align*}
&	\|\wt{w}_\psi\|_D+\|\wt{v}_\psi\|_{H^1(0,\infty; H^2(\Omega))} \\& \le M_3 \Big( \|h_1(\psi)\|_{L^2(0,\infty;L^2(\Omega))}+ \|g_1(\psi)\|_{L^2(0,\infty;L^2(\Omega))} + \frac{\kappa}{\lambda}\|y_\infty\|_{H^2(\Omega)}+ \|w_0\|_{H^1_0(\Omega)}  \Big) \\
	& \le M_3 M_4\Big(\|\psi\|_D^{\delta+1} + \|y_\infty\|_{H^2(\Omega)} \|\psi\|_D^\delta + \|y_\infty\|_{H^2(\Omega)}^{\delta}\|\psi\|_D+ \|y_\infty\|_{L^{3\delta}(\Omega)}^{\delta}\|\psi\|_D \\
	&\qquad +  \|\psi\|_D^{2\delta+1} + \|y_\infty\|_{L^{4\delta}(\Omega)}^{2\delta}\|\psi\|_D+\|\psi\|_D^{\delta+1}+\|y_\infty\|_{L^{2\delta}(\Omega)}^\delta \|\psi\|_D \Big)\\
	& \quad + M_3 \Big( \frac{\kappa}{\lambda}\|y_\infty\|_{H^2(\Omega)}  + \|w_0\|_{H^1_0(\Omega)}  \Big)\\
	& \le M_3 M_4\Big(2\|\psi\|_D^{\delta+1} + \|y_\infty\|_{H^2(\Omega)} \|\psi\|_D^\delta + \|\psi\|_D^{2\delta+1} \\
	&\qquad +   (\|y_\infty\|_{H^2(\Omega)}^{\delta}+ \|y_\infty\|_{L^{3\delta}(\Omega)}^{\delta}+\|y_\infty\|_{L^{2\delta}(\Omega)}^\delta)\|\psi\|_D+ \|y_\infty\|_{L^{4\delta}(\Omega)}^{2\delta}\|\psi\|_D \Big)\\
	& \quad + M_3 \Big( \frac{\kappa}{\lambda}\|y_\infty\|_{H^2(\Omega)}  + \|w_0\|_{H^1_0(\Omega)}  \Big). 
\end{align*}
Now, for all 
\begin{align*}
	& \left\lbrace \|y_\infty\|_{H^2(\Omega)}, \|w_0\|_{H^1_0(\Omega)} \right\rbrace \le \mathfrak{M}_1 \rho \, \text{ and }  \rho \le \mathfrak{M}_1, 
\end{align*}
where 
\begin{align*}
	\mathfrak{M}_1 = \min \left\lbrace \frac{1}{(10M_3M_4)^{\frac{1}{\delta}}},\frac{1}{(10M_3M_4)^{\frac{1}{\delta+1}}}, \frac{1}{(10M_3M_4)^{\frac{1}{2\delta}}}, \frac{1}{(10M_3M_4)^{\frac{1}{4\delta}}}, \frac{\lambda}{10M_3\kappa}, \frac{1}{9M_3} \right\rbrace,
\end{align*}
we obtain 
\begin{align*}
	\|\wt{w}_\psi\|_D+\|\wt{v}_\psi\|_{H^1(0,\infty; H^2(\Omega))} & \le M_3M_4 \left( 2\rho^\delta + \mathfrak{M}_1\rho^\delta+\rho^{2\delta} +2\mathfrak{M}_1^\delta \rho^\delta + \mathfrak{M}_1^{2\delta}\rho^{2\delta} \right)\rho + \frac{1}{5}\rho\\
	& \le \rho.
\end{align*}
Let us define the maps $$S_1:D_\rho \to D_\rho \times B_\rho \ \text{ by }\  S_1(\psi)=(\wt{w}_\psi, \wt{v}_\psi)$$  and $$S_2: D_\rho \times B_\rho \to D_\rho\ \text{ by }\ S_2(\wt{w}_\psi, \wt{v}_\psi)=\wt{w}_\psi.$$ Then the map $$S:=S_2 \circ S_1: D_\rho \to D_\rho$$ is well defined for suitable $\rho.$ 
Next, our aim is to show that the map $S$ is a contraction.
For given $\psi_1, \psi_2 \in D,$ let $(\wt{w}_{\psi_1},\wt{v}_{\psi_1})$ and $(\wt{w}_{\psi_2},\wt{v}_{\psi_2})$ be the corresponding solution pair. Thus, $W:=\wt{w}_{\psi_1}-\wt{w}_{\psi_2}$ and $V:=\wt{v}_{\psi_1}- \wt{v}_{\psi_2}$ satisfy
\begin{equation}  \label{eq:ff-nh-nlclslpsys-diff}
\left\{
	\begin{aligned}
		&  W_t -\eta \Delta W +(\beta\gamma -\nu) W  - \kappa\Delta V  = \chi_{\mathcal{O}} K(W, V) - (h_1(\psi_1)-h_1(\psi_2))  \no\\&\qquad- (g_1(\psi_1)-g_1(\psi_2))\ \text{ in }\  \Omega \times (0,\infty),\\
		& V_t+\lambda V -\nu V - W= 0 \text{ in }\  \Omega \times (0,\infty), \\
		&   W(x,t)=0, \quad V(x,t)= 0 \,  \text{ for all } \,(x,t)\in \Gamma\times (0,\infty),\\
		& W(x,0)=0, \quad V(x,0)=0 \text{ for all } x\in \Omega.
	\end{aligned}
\right.
\end{equation}
Again applying Theorem \ref{th:reg-nh-ff-Psi} followed by Proposition \ref{pps-h1g1-lip}, we have
\begin{align*}
&	\|W\|_D+\|V\|_{H^1(0,\infty;H^2(\Omega))}\\ & \le M_3 \left( \|h_1(\psi_1) - h_1(\psi_2)\|_{L^2(0,\infty; L^2(\Omega))} + \|g_1(\psi_1) - g_1(\psi_2)\|_{L^2(0,\infty; L^2(\Omega))} \right) \\
	& \le M_3 M_5 \|\psi_1 -\psi_2\|_D \Big( \|\psi_1 \|_D^\delta+\|\psi_2 \|_D^\delta +\|\psi_1\|_D\|\psi_2\|_D^{\delta-1}+ (\|\psi_1\|_D^{\delta-1} + \|\psi_2\|_D^{\delta-1})\|y_\infty\|_{H^2(\Omega)}  \\
	& \qquad+\|y_\infty\|_{H^2(\Omega)}^\delta + \|\psi_1\|_D \|y_\infty\|_{L^{6(\delta-1)}(\Omega)}^{\delta-1} + \|y_\infty\|_{L^{3\delta}(\Omega)}^{\delta} +\|\psi_1\|_D^{2\delta}+ \|\psi_2\|_D^{2\delta} + \|y_\infty\|_{L^{6\delta}(\Omega)}^{2\delta} \\
	& \qquad  + \|\psi_1\|_D^{\delta}+ \|\psi_2\|_D^{\delta} + \|y_\infty\|_{H^2(\Omega)}^{\delta}  \Big) .
\end{align*}
Now, for all 
\begin{align*}
	& \left\lbrace \|y_\infty\|_{H^2(\Omega)}, \|w_0\|_{H^1_0(\Omega)} \right\rbrace \le \mathfrak{M}_2 \rho  \, \text{ and }  \rho \le \mathfrak{M}_2, 
\end{align*}
where 
\begin{align*}
	\mathfrak{M}_2 = \min \left\lbrace \frac{1}{(15M_3M_5)^{\frac{1}{\delta}}},\frac{1}{(15M_3M_5)^{\frac{1}{\delta+1}}}, \frac{1}{(15M_3M_5)^{\frac{1}{2\delta-1}}}, \frac{1}{(15M_3M_5)^{\frac{1}{2\delta}}}, \frac{1}{(15M_3M_5)^{\frac{1}{4\delta}}}, \right\rbrace,
\end{align*}
we have 
\begin{align*}
&	\|W\|_D+\|V\|_{H^1(0,\infty;H^2(\Omega))} \nonumber\\& \le M_3 \left( \|h_1(\psi_1) - h_1(\psi_2)\|_{L^2(0,\infty; L^2(\Omega))} + \|g_1(\psi_1) - g_1(\psi_2)\|_{L^2(0,\infty; L^2(\Omega))} \right) \\
	& \le  \|\psi_1 -\psi_2\|_D M_3 M_5 \left( 5\rho^\delta + 2\mathfrak{M}_2 \rho^\delta + \mathfrak{M}_2^{\delta-1}\rho^\delta + 2\mathfrak{M}_2^\delta \rho^\delta  + 2 \rho^{2\delta}+ \mathfrak{M}_2^{2\delta}\rho^{2\delta}\right)\\
	& \le \frac{13}{15} \|\psi_1 -\psi_2\|_D.
\end{align*}
	Thus, by choosing 
	\begin{align*}
	& \rho_3=\overline{M}=\min \left\lbrace \mathfrak{M}_1, \mathfrak{M}_2 \right\rbrace,
	\end{align*}
		 we obtain that $S$ is a contraction and self map on $D_\rho$ for all $0<\rho\le \rho_3.$ Hence, by using the Banach fixed point theorem, we have a fixed point of $S$ and hence  a unique solution $(\wt{w}, \wt{v})$ of \eqref{eq:ff-nlclslpsys} satisfying 
	\begin{align*}
		\|\wt{w}\|_{L^2(0,\infty;H^2(\Omega))}^2 & +\|\wt{w}\|_{L^\infty(0,\infty;H^1_0(\Omega))}^{2}+\|\wt{w}\|_{H^1(0,\infty;L^2(\Omega))}^2 \\
		& +\|\wt{w}\|_{L^{2(\delta+1)}(0,\infty;L^{6(\delta+1)}(\Omega))}^{2}+\|\wt{v}\|_{H^1(0,\infty;H^2(\Omega))}^2 \le 2\rho^2,
	\end{align*}
	for all $0<\rho\le \rho_3,$ $\|w_0\|_{H^1_0(\Omega)}\le \overline{M} \rho,$ $\|y_\infty\|_{H^2(\Omega)}\le \overline{M} \rho,$ where $\overline{M}$ and $\rho_3$ are as mentioned above. 
\end{proof}

Now, using the above theorem, we conclude the section by stating stabilization result for the integral equation \eqref{eq:GBHENC}. 
\begin{Theorem}\label{thm-main-1}
	Let $\nu \in (0,\lambda)$ be arbitrary. There exist a continuous linear operator $\mathbf{K} \in \mathcal{L} (L^2(0,\infty; L^2(\Omega)), L^2(\Omega)),$ $\rho_3>0,$ $\overline{M}>0$ such that for all $0<\rho<\rho_3,$ $y_\infty \in H^2(\Omega)\cap H^1_0(\Omega)$ solution of \eqref{eq:GBHE-ST}, and $y_0\in H^1_0(\Omega)$ satisfying 
	\begin{align*}
		\|y_\infty\|_{H^2(\Omega)}\le \overline{M} \rho \text{ and } \|y_0-y_\infty\|_{H^1_0(\Omega)}\le \overline{M} \rho,
	\end{align*}
	the closed loop system
	{\small 
		\begin{equation} 
			\left\{
			\begin{aligned}
				& y_t - \eta \Delta y +  \alpha y^\delta \sum_{i=1}^d \frac{\partial y}{\partial x_i} - \kappa \int_0^t e^{-\lambda (t-s)}\Delta y(\cdot,s)ds + \beta y(y^\delta-1)(y^\delta-\gamma) =f_\infty + \mathbf{K}y \ \text{ in }\  \Omega\times (0,\infty), \\
				& y=0  \ \text{ on }\  \Gamma\times (0,\infty), \\
				& y(x,0)=y_0 \ \text{ in }\ \Omega,
			\end{aligned}\right.
	\end{equation}}
admits a unique solution $y$ satisfying
\begin{align*}
	&\|y(\cdot,t)-y_\infty\|_{H^1_0(\Omega)} + \left\| \int_0^t e^{-\lambda(t-s)} (y(\cdot,s) - y_\infty)  \right\|_{H^2(\Omega)} \\
	& \qquad \qquad \le C e^{-\nu t} \left( \|y_\infty\|_{H^2(\Omega)}+ \|y_0-y_\infty\|_{H^1_0(\Omega)} \right) \text{ for all } t>0,
\end{align*}
and for some $C>0.$
\end{Theorem}
The proof follows from Theorem \ref{th:stabIneg-NC} by defining
\begin{align*}
	\mathbf{K}y(\cdot,t)=K \left(y(\cdot,t), \int_0^t e^{-\lambda(t-s)}y(\cdot,s)ds\right),
\end{align*} 
where $K$ is as in Remark \ref{rem:feedbackop}.

\begin{Remark}
	The main difference between Theorems \ref{th:stabIneg-zero} and  \ref{thm-main-1} is the range of $\nu$. In the zero state case, $\nu\in(0,\lambda+\frac{\kappa}{\eta})$ and for the non-stationary case, $\nu\in(0,\lambda)$. 	One can observe that the decay rate in the case of stabilization around the zero steady state is bounded by $ \nu_0 $ (see Theorems \ref{th:stabcoupl-zero} - \ref{th:stabIneg-zero}), which is larger than the bound on the decay rate in the case of stabilization around the non-constant steady state (see Theorem \ref{th:stabIneg-NC}). This discrepancy arises due to the presence of the term $e^{(\nu -\lambda)t}\Delta y_\infty $ in \eqref{eq:LinearizedCouple-NC}. As per our approach, this term must belong to $ L^2(0,\infty; L^2(\Omega)) $, which slightly reduces the achievable decay rate in this case.
\end{Remark}

\section{Numerical simulations} \label{sec:NS}
This section presents numerical computations to validate the theoretical results. We examine an example where the uncontrolled solution is unstable and demonstrate how computing the feedback control stabilizes it. The following two subsections explore both cases: stabilizability around the zero steady state and a non-constant steady state.

\subsection{Simulation for zero steady state case} This subsection is devoted to conduct a numerical example by considering a GBHE equation with memory. The computational domain we consider is $\Omega=(0,1)\times (0,1)$ and this is discretized  using a $P_1$ finite element method on a structured triangular mesh. We consider the example \eqref{eq:GBHE-u-coup-Z} with the following data:
\begin{align*}
	\eta=0.2, \, \alpha=1, \, \delta=1, \, \kappa=1.5, \, \beta=1.5, \,  \gamma = 0.5, \, \lambda=3,
\end{align*}
and $y_0(x)=x_1(1-x_1)x_2(1-x_2).$ Using finite element method, the ODE system corresponding to the semi-discretized system is
\begin{align*}
	E_N\frac{d}{dt}Y_N(t)= A_N Y_N(t)+B_N u_N(t)+F_N(Y_N(t)) \ \text{ in }\  (0,1)\times (0,1), \quad Y_N(0)=Y_{N_0},
\end{align*}
where  is
\begin{align*}
	& E_N=\begin{pmatrix} M_N & 0 \\ 0 & M_N \end{pmatrix}, \quad A_N = \begin{pmatrix} -\eta K_N & -\kappa K_N \\  -\lambda M_N & M_N \end{pmatrix}, \quad Y_N(t)=\begin{pmatrix} y_N(t) \\ z_N(t) \end{pmatrix}, \quad B_N= \begin{pmatrix} M_N \\ 0 \end{pmatrix},
\end{align*} 
with $M_N \in \mathbb{R}^{N\times N}$ is the symmetric mass matrix, $K_N \in \mathbb{R}^{N\times N}$ is the stiffness matrix, $B_N$ is the control matrix, and $F_N$ is matrix corresponding to  the nonlinear terms. Similarly, we first consider the shifted linear ODE as
\begin{align*}
	E_N\frac{d}{dt}\wt{Y}_N(t)= (A_N +\nu I_N) \wt{Y}_N(t)+B_N \wt{u}_N(t) \ \text{ in }\  (0,1)\times (0,1), \quad \wt{Y}_N(0)=Y_{N_0}.
\end{align*}
Before checking stabilizability, we first analyze the eigenvalues of the system. Specifically, we compute the approximate eigenvalues of the shifted principal operator \( A_\nu \) by computing the eigenvalues of $E_N^{-1}(A_N +\nu I_N)$  for the given parameter choice and observe that two of them have positive real parts (see Figure~\ref{fig:combined}(A)), indicating instability in the uncontrolled linear system (see Figure~\ref{fig:combined-LinSol}(A)). Next, we compute the feedback operator by solving corresponding algebraic Riccati equation (in matrix form):
\begin{equation}\label{eq:matrixRiccati}
\begin{aligned}
&	P_N (E_N^{-1}(A_N +\nu I_N)) + (A_N +\nu I_N)^TE_N^{-1}P_N- P_N E_N^{-1}B_NE_N^{-1}B_N^TE_N^{-1}P_N+E_N=0, \\
& P_N^T=P_N\ge 0,
\end{aligned}
\end{equation}
see \cite[(8.5)]{WKRPCE} or \cite[Section 3.4]{WKR}. In this case, we obtain a feedback control as 
\begin{align*}
	\widetilde{u}_N(t)=-E_N^{-1}B_N^TE_N^{-1}P_NY_N(t),
\end{align*}
where $\widetilde{Y}_N(t)$ solves 
\begin{align}\label{eq:NS-closdLinMat}
E_N\frac{d}{dt}\wt{Y}_N(t)= \left((A_N +\nu I_N) -B_N S_NP_N \right) \wt{Y}_N(t)\ \text{ in }\  (0,1)\times (0,1), \quad \wt{Y}_N(0)=Y_{N_0}.
\end{align}
Here, $S_N$ can be computed from 
\begin{align*}
	E_N S_N=\widehat{B}_N^T \text{ and } \widehat{B}_N \text{ solves } E_N \widehat{B}_N=B_N.
\end{align*}
Next, we compute the approximated eigenvalues of the perturbed operator \( A + BK \), where the perturbation is introduced by the feedback operator \( K \). Specifically, we compute the eigenvalues of \( E_N^{-1} \left( (A_N + \nu I_N) - B_N S_N P_N \right) \) and observe that all eigenvalues have negative real parts, as shown in Figure~\ref{fig:combined}(B). This indicates that the system exhibits a stable solution. Furthermore, we compute the solution of \eqref{eq:NS-closdLinMat} and plot its energy over time, as depicted in Figure~\ref{fig:combined-LinSol}(B). The energy is observed to decrease over time, confirming that the solution is stabilized. Notably, the previously unstable eigenvalues shift into the negative half-plane after applying feedback control, as shown in Figure~\ref{fig:combined}(B), further demonstrating the effectiveness of the control strategy.

\begin{figure}[h]
	\centering
	\begin{subfigure}{0.48\textwidth}
		\centering
		\includegraphics[width=\textwidth]{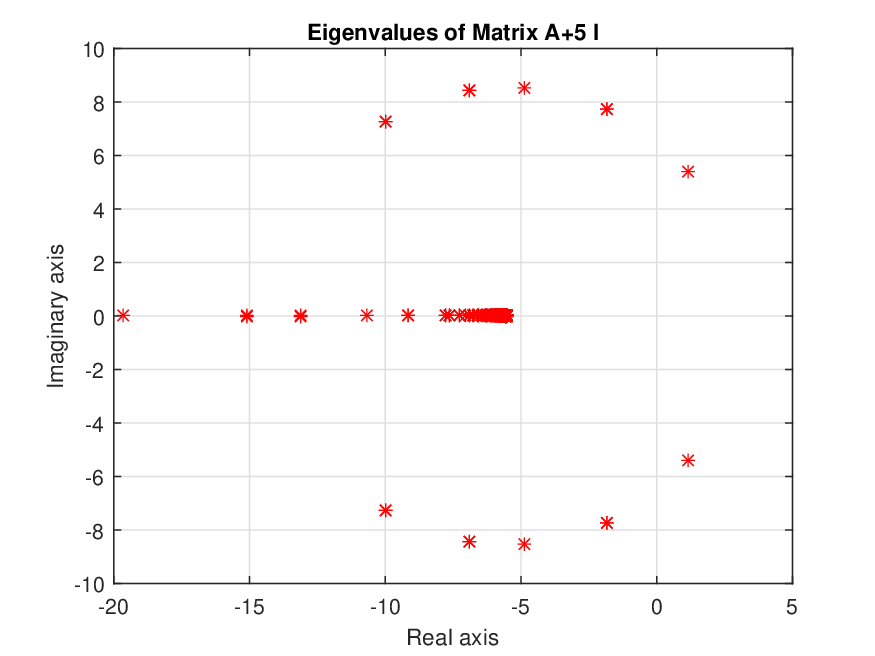}  
		\caption{Before stabilization}
		\label{fig:eigBS}
	\end{subfigure}
	\hfill
	\begin{subfigure}{0.48\textwidth}
		\centering
		\includegraphics[width=\textwidth]{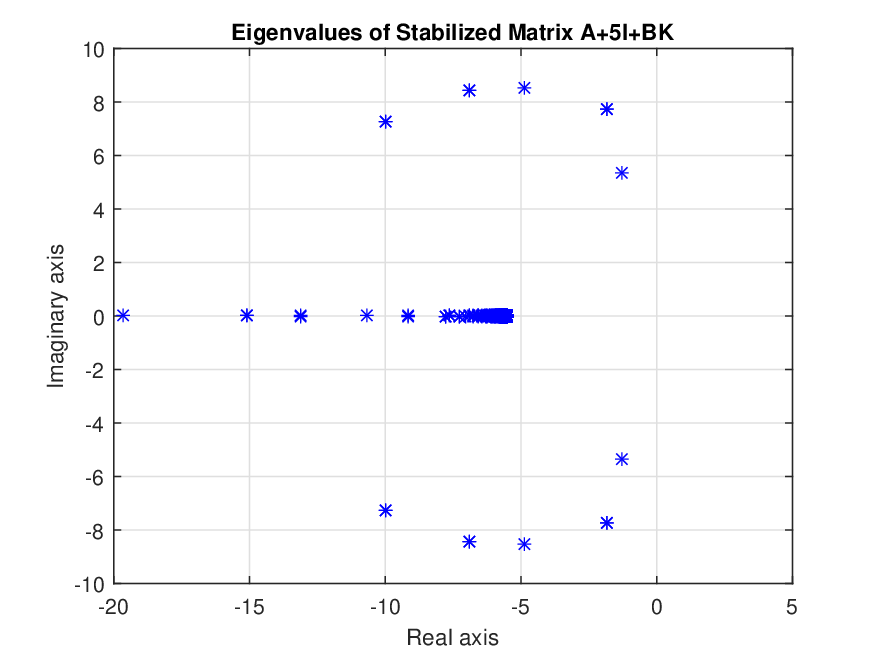}  
		\caption{After stabilization}
		\label{fig:eigAS}
	\end{subfigure}
	
	\caption{Eigenvalues of the shifted principal operator $A+\nu I$ and after stabilization $A+\nu I +BK$}
	\label{fig:combined}
\end{figure}

\begin{figure}[h]
	\centering
	\begin{subfigure}{0.48\textwidth}
		\centering
		\includegraphics[width=\textwidth]{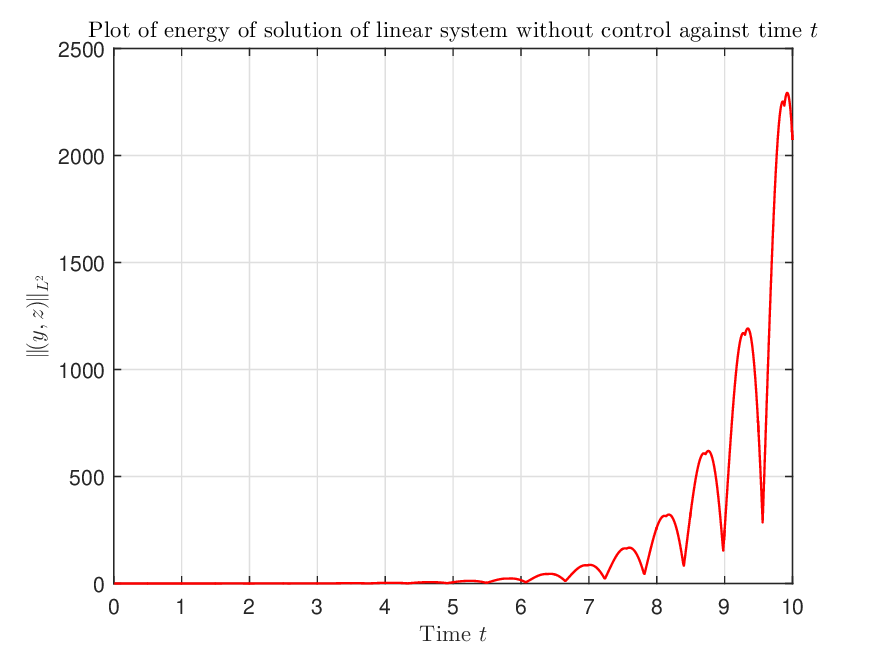}  
		\caption{Before stabilization}
		\label{fig:LinSolBS}
	\end{subfigure}
	\hfill
	\begin{subfigure}{0.48\textwidth}
		\centering
		\includegraphics[width=\textwidth]{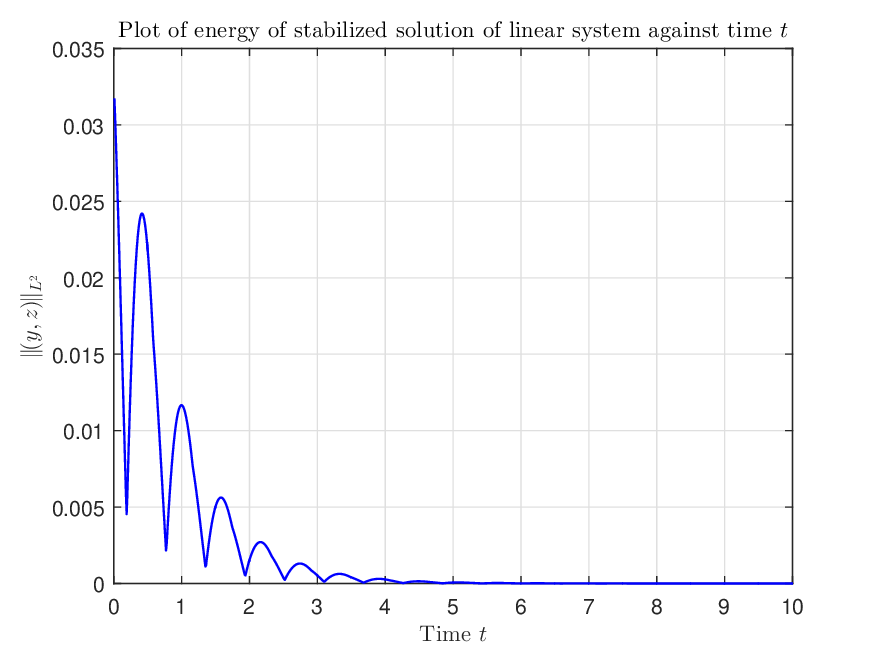}  
		\caption{After stabilization}
		\label{fig:LinStabSolAS}
	\end{subfigure}
	
	\caption{Plot of norm of solution of linear system before and after stabilization}
	\label{fig:combined-LinSol}
\end{figure}

\begin{figure}[h]
	\centering
	\begin{subfigure}{0.48\textwidth}
		\centering
		\includegraphics[width=\textwidth]{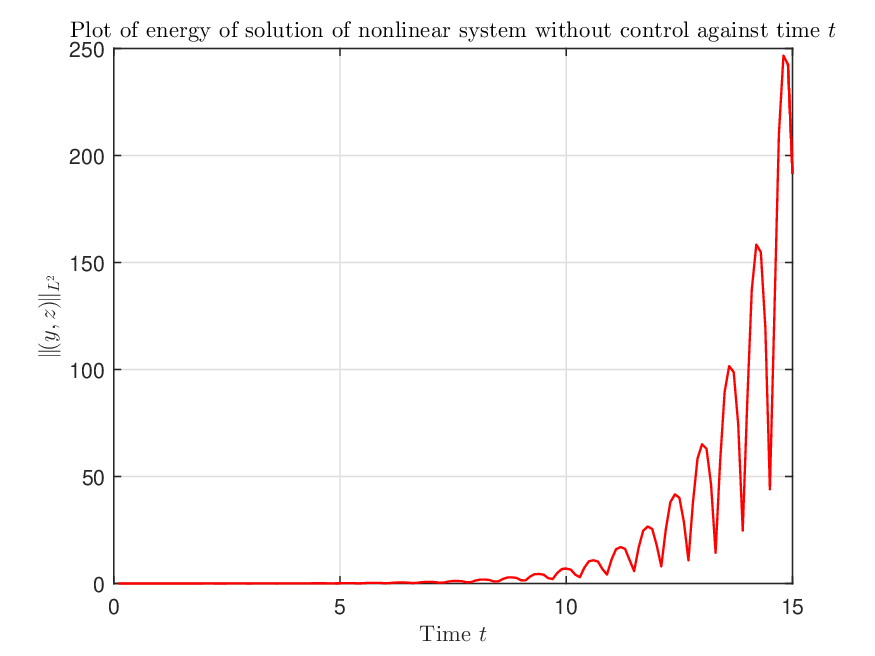}  
		\caption{Before stabilization}
		\label{fig:NonLinSolBS}
	\end{subfigure}
	\hfill
	\begin{subfigure}{0.48\textwidth}
		\centering
		\includegraphics[width=\textwidth]{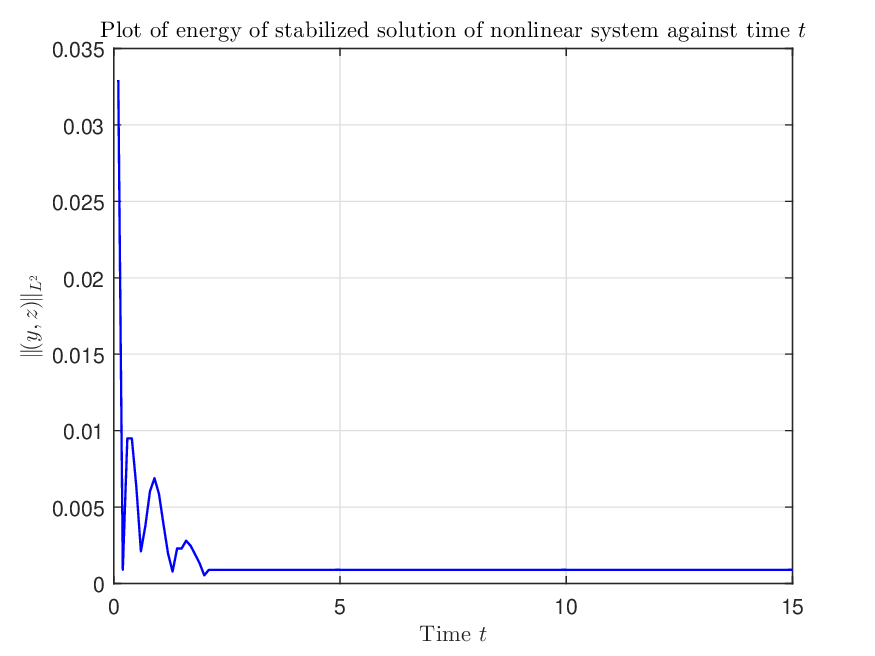}  
		\caption{After stabilization}
		\label{fig:NonLinStabSolAS}
	\end{subfigure}
	
	\caption{Plot of norm of solution of full non-linear system before and after stabilization}
	\label{fig:combined-NonLinSol}
\end{figure}

Now, we move to the case of the nonlinear system. First, we compute the solution of the uncontrolled system
\begin{align*}
	E_N \frac{d}{dt} \wt{Y}_N(t) = (A_N + \nu I_N) \wt{Y}_N(t) \widetilde{F}_N(\widetilde{Y}_N)(t) \quad \text{in } (0,1) \times (0,1), \quad \wt{Y}_N(0) = Y_{N_0},
\end{align*}
using Newton's method, as discussed in \cite[Section 4.3]{WKR}. We plot the energy of the solution of the uncontrolled system over time in Figure~\ref{fig:combined-NonLinSol}(A), where we observe that the energy increases with time, indicating that the solution is unstable. Next, we compute the feedback operator, following the same approach as for the linear system. Specifically, we solve the closed-loop system
\begin{align*}
	E_N \frac{d}{dt} \wt{Y}_N(t) = \left( (A_N + \nu I_N) - B_N S_N P_N \right) \wt{Y}_N(t) + \widetilde{F}_N(\widetilde{Y}_N)(t) \quad \text{in } (0,1) \times (0,1), \quad \wt{Y}_N(0) = Y_{N_0},
\end{align*}
and plot the energy over time in Figure~\ref{fig:combined-NonLinSol}(B). As shown, the solution stabilizes, confirming the effectiveness of the feedback control.

\subsection{Simulation around non-constant steady state}
To consider an example similar to the previous subsection, we take the computational domain as \( \Omega = (0,1) \times (0,1) \), which is discretized using a \( P_1 \) finite element method on a structured triangular mesh. For this example, we consider the shifted system \eqref{eq:GBHE-Lin-w-y_inft} as \eqref{eq:GBHE-NLClsCoupShifNCST} with the following data:
\begin{align*}
	\eta = 0.2, \, \alpha = 1, \, \delta = 1, \, \kappa = 1.5, \, \beta = 1.5, \, \gamma = 0.5, \, \lambda = 3,
\end{align*}
and \( y_0(x_1, x_2) = x_1(1 - x_1) x_2(1 - x_2) \), \( y_\infty(x_1, x_2) = \sin(\pi x_1) \sin(\pi x_2) \). The matrix form of the system is written as:
{\small
\begin{align*}
	& E_N \frac{d}{dt} \wt{W}_N(t) = (A_N + \nu I_N) \wt{W}_N(t) + \widetilde{H}_N(\widetilde{Y}_N)(t) + \widetilde{G}_N(\widetilde{Y}_N)(t) + \widetilde{Y}_{\infty,N}(\widetilde{Y}_N) + B_N \wt{u}_N \quad \text{in } (0,1) \times (0,1), \\
	& \wt{W}_N(0) = Y_{N_0} - Y_{\infty,N_0}.
\end{align*}}
Here, the matrices \( \widetilde{G}_N(\widetilde{Y}_N) \) and \( \widetilde{H}_N(\widetilde{Y}_N) \) correspond to the nonlinear terms \( h_1 \) and \( g_1 \) (see \eqref{eq:g1} - \eqref{eq:g1}), and \( \widetilde{Y}_{\infty,N}(\widetilde{Y}_N) \) corresponds to \( e^{(\nu - \lambda)t} \Delta y_\infty \).  Using Newton's method and following the procedure given in \cite[Section 4.5]{WKR}, we first compute the solution of the uncontrolled system:
\begin{align*}
	& E_N \frac{d}{dt} \wt{W}_N(t) = (A_N + \nu I_N) \wt{W}_N(t) + \widetilde{H}_N(\widetilde{Y}_N)(t) + \widetilde{G}_N(\widetilde{Y}_N)(t) + \widetilde{Y}_{\infty,N}(\widetilde{Y}_N) \quad \text{in } (0,1) \times (0,1), \\
	& \wt{W}_N(0) = Y_{N_0} - Y_{\infty,N_0},
\end{align*}
and plot its energy over time in Figure~\ref{fig:combined-NonLinSolNZ}(A) to confirm instability. Next, we compute the feedback operator by solving the algebraic Riccati equation \eqref{eq:matrixRiccati} in matrix form and obtain the closed-loop system:
\begin{align*}
	& E_N \frac{d}{dt} \wt{W}_N(t) = \left( (A_N + \nu I_N) - B_N S_N P_N \right) \wt{W}_N(t) + \widetilde{H}_N(\widetilde{Y}_N)(t) \\
	& \hspace{5cm} +  \widetilde{G}_N(\widetilde{Y}_N)(t) + \widetilde{Y}_{\infty,N}(\widetilde{Y}_N) \quad \text{in } (0,1) \times (0,1), \\
	& \wt{W}_N(0) = Y_{N_0} - Y_{\infty,N_0}.
\end{align*}
The energy of the solution is plotted in Figure~\ref{fig:combined-NonLinSolNZ}(B), confirming the stability of the closed-loop system.

\begin{figure}[h]
	\centering
	\begin{subfigure}{0.48\textwidth}
		\centering
		\includegraphics[width=\textwidth]{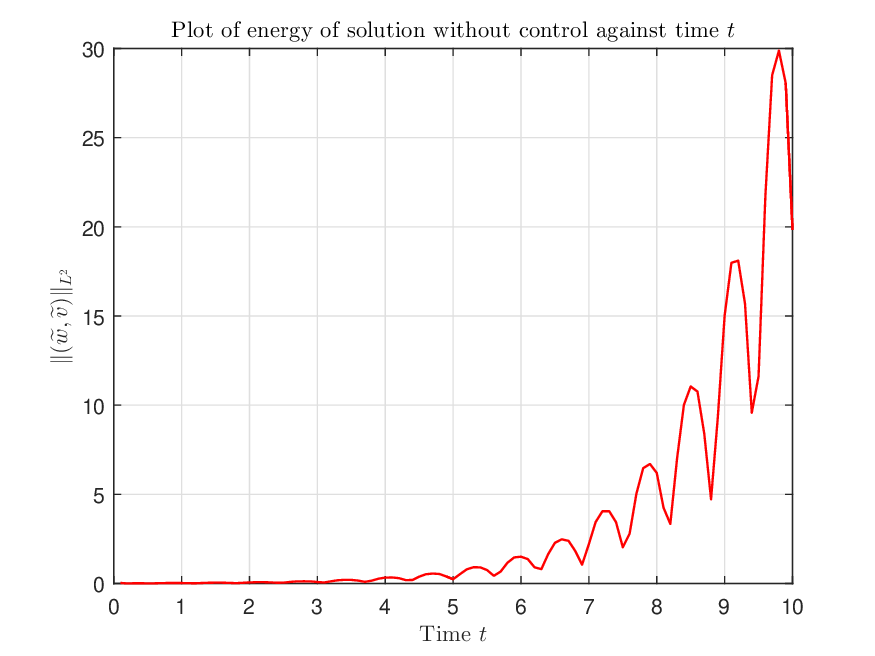}  
		\caption{Before stabilization}
		\label{fig:NonLinSolBSNZ}
	\end{subfigure}
	\hfill
	\begin{subfigure}{0.48\textwidth}
		\centering
		\includegraphics[width=\textwidth]{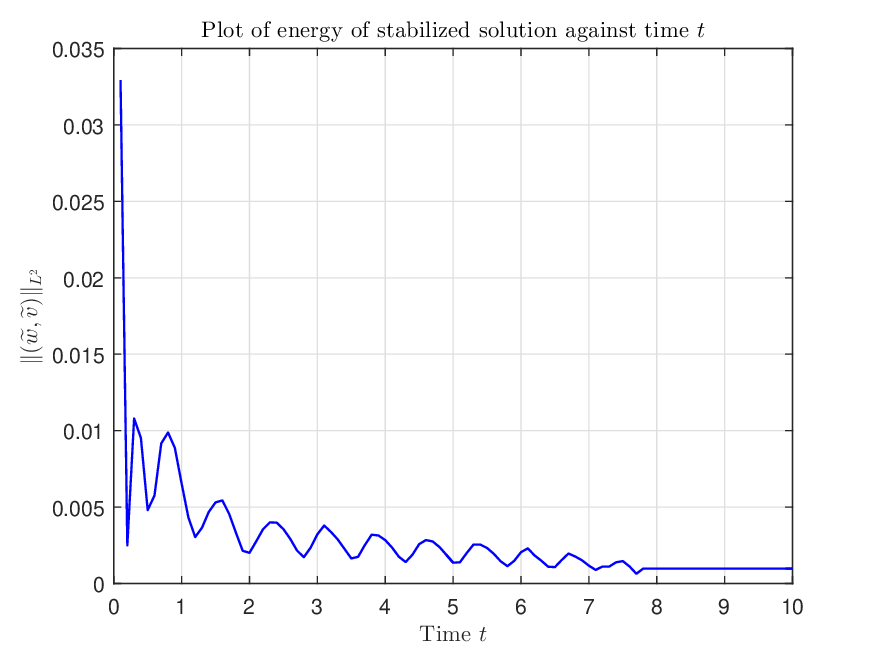}  
		\caption{After stabilization}
		\label{fig:NonLinStabSolASNZ}
	\end{subfigure}
	
	\caption{Plot of norm of solution of full non-linear system before and after stabilization}
	\label{fig:combined-NonLinSolNZ}
\end{figure}

\appendix
\section{Proof of Proposition \ref{pps:heatStrongL2L6}} \label{app:heat}

\begin{proof}[Proof of Proposition \ref{pps:heatStrongL2L6}]
	It is well-known that for $w_0\in H^1_0(\Omega)$ and $f\in L^2(0,\infty; L^2(\Omega)),$ the given heat equation has a unique strong solution $w \in L^2(0,\infty; H^2(\Omega)) \cap L^\infty(0,\infty; H^1_0(\Omega)) $ with  $w_t \in L^2(0,\infty; L^2(\Omega))$ (see \cite[Theorem 5, Section 7.1]{Eva}) satisfying 
	\begin{align*}
		\|w\|_{L^\infty(0,\infty; H^1_0(\Omega))} + \|w\|_{L^2(0,\infty; H^2(\Omega))}+\|w\|_{H^1(0,\infty;L^2(\Omega))} \le  C\left( \|w_0\|_{H^1_0(\Omega)} +\|f\|_{L^2(0,\infty;L^2(\Omega))}  \right),
	\end{align*}
	for some  constant $C>0.$ For the values of $\delta$ given in \eqref{eqdef:delta}, by the Sobolev embedding, we know that $H_0^1(\Omega)\hookrightarrow L^{2(\delta+1)}(\Omega)$. Now, it remains to show that  $w\in L^{2(\delta+1)}(0,\infty; L^{6(\delta+1)}(\Omega))$. The following calculations can be justified by using a proper Faedo-Galerkin approximation scheme.  Therefore, we formally take an inner product in the heat equation with $|w|^{2\delta}w$ and observe
	\begin{align}\label{eq:wkform_p-2}
		& \left( \frac{dw}{dt}, |w|^{2\delta}w\right)-\eta \left( \Delta w, |w|^{2\delta}w\right)  = \left(f , |w|^{2\delta}w\right).
	\end{align}
	Now, using integration by parts and employing simple manipulations, we have
	\begin{align}\label{1st_est}
		\frac{d}{dt}\|w(t)\|_{L^{2(\delta+1)}}^{2(\delta+1)} & = \frac{d}{dt} \int_\Omega |w(t,x)|^{2(\delta+1)}dx = \int_\Omega \frac{d}{dt} \left( |w(t,x)|^2\right)^{{(\delta+1)}}dx\no \\
		&  = {(\delta+1)} \int_\Omega  \left( |w(t,x)|^{2}\right)^{\delta} \frac{d}{dt}|w(t,x)|^{2} dx   \no \\
		& = 2(\delta+1)\int_\Omega   \frac{dw(t,x)}{dt} |w(t,x)|^{2\delta}w(t,x) dx\no\\
		& =2(\delta+1) \left( \frac{dw(t,x)}{dt}, |w(t,x)|^{2\delta}w(t,x)\right).
	\end{align}
	Similarly, it can be verified that 
	\begin{align}\label{2nd_est}
		\left( -\Delta w, |w|^{2\delta}w \right) & = -\int_\Omega \Delta w |w|^{2\delta}w dx = -\sum_{i=1}^d \int_\Omega \frac{\partial^2 w}{\partial x_i^2} |w|^{2\delta}w dx \no\\
		& =\sum_{i=1}^d\int_\Omega \frac{\partial w}{\partial x_i} \left( |w|^{2\delta} \frac{\partial w}{\partial x_i} + 2\delta  (|w|^2)^{\delta-1}  w^2 \frac{\partial w}{\partial x_i}\right) dx \no\\
		& =(2\delta+1) \int_\Omega |w|^{2\delta} \sum_{i=1}^d\left(\frac{\partial w}{\partial x_i}\right)^2 dx = (2\delta+1) \||w|^{\delta}\nabla w\|^2 .
	\end{align}
	Now, we estimate right hand side of \eqref{eq:wkform_p-2}, using integration by parts and Cauchy Schwartz inequality as follows:
	\begin{align}\label{f_est}
		\left\vert\left(f , |w|^{2\delta}w\right)\right\vert & = \left\vert\left(|w|^{\delta}f , |w|^{\delta}w\right)\right\vert  \le \| \nabla(|w|^{\delta}w)\| \| |w|^{\delta}f\|_{H^{-1}(\Omega)} \no\\
		& \le (\delta+1)\||w|^{\delta}\nabla w\| C \| |w|^{\delta} f\|_{L^{\frac{2\delta+2}{2\delta+1}}(\Omega)} \no\\
		& \le C (\delta+1)\||w|^{\delta}\nabla w\|  \| |w|^{\delta} \|_{L^{\frac{2(\delta+1)}{\delta}}(\Omega)}\|f\| \no\\
		& \le C (\delta+1) \||w|^{\delta}\nabla w\| \|w\|_{L^{2(\delta+1)}}^{\delta}\|f\| \le \frac{(2\delta+1) \eta}{2} \||w|^{\delta}\nabla w\|^2 + C\|w\|_{L^{2(\delta+1)}(\Omega)}^{2\delta}\|f\|^2. 
	\end{align}
	Using the estimates\eqref{1st_est}, \eqref{2nd_est}, and \eqref{f_est} in \eqref{eq:wkform_p-2}, we obtain 
	\begin{align*}
		\frac{1}{2(\delta+1)}\frac{d}{dt}\|w(t)\|_{L^{2(\delta+1)}(\Omega)}^{2(\delta+1)}+ \frac{(2\delta+1)\eta}{2} \||w|^{\delta}\nabla w\|^2 \le C\|w\|_{L^{2(\delta+1)}(\Omega)}^{2\delta}\|f\|^2.
	\end{align*}
	For $\zeta(t):=\|w(t)\|_{L^{2(\delta+1)}(\Omega)}^{2(\delta+1)},$ the above relation can be expressed as 
	\begin{align*}
		\zeta(t) \le \zeta(0)+C\int_0^t \zeta(s)^{\frac{2\delta}{2(\delta+1)}}\|f(s)\|^2 ds.
	\end{align*}
	Since $0<\frac{2\delta}{2(\delta+1)}<1,$ by using a nonlinear generalization of Gronwall’s inequality (Theorem \ref{lem:NonlinGronwall}), it can be derived that
	\begin{align*}
		\zeta(t) \le \left\lbrace \zeta(0)^{\frac{1}{\delta+1}} + \frac{\delta}{\delta+1}\int_0^t \|f(s)\|^2 ds   \right\rbrace^{\delta+1},
	\end{align*}
	that is, 
	\begin{align*}
		\|w(t)\|_{L^{2(\delta+1)}(\Omega)}^{2(\delta+1)} \le \left\lbrace \|w_0\|_{L^{2(\delta+1)}(\Omega)}^2 + \frac{\delta}{\delta+1}\int_0^t \|f(s)\|^2 ds   \right\rbrace^{\delta+1}.
	\end{align*}
	Thus, we deduce for all $0\le t<\infty,$
	\begin{align*}
	&	\|w(t)\|_{L^{2(\delta+1)}(\Omega)}^{2(\delta+1)} + (2\delta+1)(\delta+1)\eta \int_0^t \||w(s)|^{\delta}\nabla w(s)\|^2 ds \\& \le C \left( \|w_0\|_{L^{2(\delta+1)}(\Omega)}^{2(\delta+1)} + \|f(t)\|_{L^2(0,\infty; L^2(\Omega))}^{2(\delta+1)} \right).
	\end{align*}
	Therefore, the solution exists for all $t\in [0,\infty),$ and $w \in C([0,\infty);L^{2(\delta+1)}(\Omega)),$ since $H_0^1(\Omega)\hookrightarrow L^{2(\delta+1)}(\Omega)$.  Now, to show that $w \in L^{2(\delta+1)}(0,\infty; L^{6(\delta+1)}(\Omega)),$ we observe that
	\begin{align*}
		\|w\|_{L^{2(\delta+1)}(0,\infty; L^{6(\delta+1)}(\Omega))}^{2(\delta+1)} & = \int_0^\infty \|w(t)\|_{L^{6(\delta+1)}(\Omega)}^{2(\delta+1)} dt  = \int_0^\infty \|(w(t))^{\delta+1}\|_{L^6(\Omega)}^2 dt \\
		& \le C(\delta+1)^2 \int_0^\infty \| (w(t))^{\delta} \nabla w(t)\|^2 dt \\
		& \le  C(\delta+1)^2 \left( \|w_0\|_{L^{2(\delta+1)}(\Omega)}^{2(\delta+1)} + \|f(t)\|_{L^2(0,\infty; L^2(\Omega))}^{2(\delta+1)} \right),
	\end{align*}
	which  completes the proof.
\end{proof}

\bibliographystyle{amsplain}
\bibliography{References}

\end{document}